 \documentclass[11pt]{article}

\usepackage{amsmath,amsthm,amssymb,xypic,stmaryrd, graphicx, mathtools}
\usepackage{hyperref}

\theoremstyle{plain}
    \newtheorem{theorem}{Theorem}[section]
    \newtheorem{lemma}[theorem]{Lemma}
    \newtheorem{corollary}[theorem]{Corollary}
    \newtheorem{proposition}[theorem]{Proposition}

 \theoremstyle{definition}
    \newtheorem{definition}[theorem]{Definition}
    
    \newtheorem{example}[theorem]{Example}

    \newtheorem{remark}[theorem]{Remark}

\theoremstyle{remark}

\numberwithin{equation}{section}

 \DeclareMathOperator{\Tr}{Tr}
 \DeclareMathOperator{\tr}{tr}

\DeclareMathOperator{\Ad}{Ad}

\DeclareMathOperator{\ind}{index}

\DeclareMathOperator{\End}{End}

\DeclareMathOperator{\sgn}{sgn}
\DeclareMathOperator{\ch}{ch}
\DeclareMathOperator{\Todd}{Todd}

\DeclareMathOperator{\rank}{rank}
\DeclareMathOperator{\sign}{sign}
\DeclareMathOperator{\diam}{diam}
\DeclareMathOperator{\fp}{fp}

\DeclareMathOperator{\Spin}{Spin}

\DeclareMathOperator{\SO}{SO}
\DeclareMathOperator{\SL}{SL}
\DeclareMathOperator{\GL}{GL}
\DeclareMathOperator{\SU}{SU}

 \DeclareMathOperator{\Ind}{Ind}

 \DeclareMathOperator{\even}{even}
  \DeclareMathOperator{\odd}{odd}
  \DeclareMathOperator{\DInd}{D-Ind}

\DeclareMathOperator{\vol}{vol}
\DeclareMathOperator{\area}{area}
\DeclareMathOperator{\rem}{rem}
\DeclareMathOperator{\res}{res}

\begin{document}


\newcommand{\myemph}{\emph}

\newcommand{\Spinc}{\Spin^c}

    \newcommand{\R}{\mathbb{R}}
    \newcommand{\C}{\mathbb{C}}
    \newcommand{\N}{\mathbb{N}}
    \newcommand{\Z}{\mathbb{Z}}
    \newcommand{\Q}{\mathbb{Q}}
    \newcommand{\bT}{\mathbb{T}}
    \newcommand{\bP}{\mathbb{P}}

\newcommand{\g}{\mathfrak{g}}
\newcommand{\h}{\mathfrak{h}}
\newcommand{\p}{\mathfrak{p}}
\newcommand{\kg}{\mathfrak{g}}
\newcommand{\kt}{\mathfrak{t}}
\newcommand{\ka}{\mathfrak{a}}
\newcommand{\XX}{\mathfrak{X}}
\newcommand{\kh}{\mathfrak{h}}
\newcommand{\kp}{\mathfrak{p}}
\newcommand{\kk}{\mathfrak{k}}
\newcommand{\kn}{\mathfrak{n}}
\newcommand{\km}{\mathfrak{m}}
\newcommand{\kso}{\mathfrak{so}}
\newcommand{\ksu}{\mathfrak{su}}
\newcommand{\kspin}{\mathfrak{spin}}

\newcommand{\cA}{\mathcal{A}}
\newcommand{\cE}{\mathcal{E}}
\newcommand{\calL}{\mathcal{L}}
\newcommand{\calH}{\mathcal{H}}
\newcommand{\cO}{\mathcal{O}}
\newcommand{\cB}{\mathcal{B}}
\newcommand{\cK}{\mathcal{K}}
\newcommand{\cP}{\mathcal{P}}
\newcommand{\cN}{\mathcal{N}}
\newcommand{\calD}{\mathcal{D}}
\newcommand{\cC}{\mathcal{C}}
\newcommand{\calS}{\mathcal{S}}
\newcommand{\cM}{\mathcal{M}}
\newcommand{\cU}{\mathcal{U}}

\newcommand{\cCM}{\cC}
\newcommand{\PM}{P}
\newcommand{\DM}{D}
\newcommand{\LM}{L}
\newcommand{\vM}{v}

\newcommand{\sumGam}{\textstyle{\sum_{\Gamma}} }

\newcommand{\sigDg}{\sigma^D_g}

\newcommand{\Bigwedge}{\textstyle{\bigwedge}}

\newcommand{\ii}{\sqrt{-1}}

\newcommand{\Ubar}{\overline{U}}

\newcommand{\beq}[1]{\begin{equation} \label{#1}}
\newcommand{\eeq}{\end{equation}}

\newcommand{\mattwo}[4]{
\left( \begin{array}{cc}
#1 & #2 \\ #3 & #4
\end{array}
\right)
}

\title{A higher index theorem on finite-volume locally symmetric spaces}

\author{
Hao Guo, 
Peter Hochs and 
Hang Wang
}
%
%
%

%

\date{\today}

\maketitle

\begin{abstract}
Let $G$ be a connected, real semisimple Lie group. Let $K<G$ be maximal compact, and let $\Gamma < G$ be discrete and such that $\Gamma \backslash G$ has finite volume. If the real rank of $G$ is $1$ and $\Gamma$ is torsion-free, then Barbasch and Moscovici obtained an index theorem for Dirac operators on the locally symmetric space $\Gamma \backslash G/K$. We obtain a higher version of this, using an index of Dirac operators on $G/K$ in the $K$-theory of an algebra on which the conjugation-invariant terms in Barbasch and Moscovici's index theorem define continuous traces. The resulting index theorems also apply when $\Gamma$ has torsion. The cases of these index theorems for
 traces defined by semisimple orbital integrals extend to Song and Tang's higher orbital integrals, and yield nonzero and computable results   even when $\rank(G)> \rank(K)$, or the real rank of $G$ is larger than $1$.
\end{abstract}

\tableofcontents

\section{Introduction}

Consider a connected, real semisimple Lie group $G$, a maximal compact subgroup $K<G$ and a discrete subgroup $\Gamma<G$. \emph{Locally symmetric spaces} of the form $X = \Gamma \backslash G/K$ come up in many places in mathematics, and include, for example, all hyperbolic manifolds, several moduli spaces, and the space $\SL(2, \Z) \backslash \SL(2, \R)/\SO(2)$ on which modular forms are defined. Spaces $X$ that  arise in this way are sometimes compact, but in many relevant cases (such as the example mentioned) they only have finite volume. Because of the relevance of these spaces, it is a worthwhile question how to construct and compute topological invariants for them.

Index theory is a useful tool for constructing and computing invariants of manifolds. This works particularly well if $\Gamma$ is torsion-free and $\Gamma \backslash G$ is compact, so that $X$ is a compact, smooth manifold. Then elliptic differential operators on $X$ have well-defined Atiyah--Singer indices. And by considering the action by $\Gamma$ on $G/K$ rather than just its quotient, one can construct a higher index of $\Gamma$-equivariant elliptic operators \cite{Connes94, Kasparov83, WY20}, with values in $K_*(C^*(\Gamma))$.  Pairings of this index with various cyclic cocycles were computed by Connes and Moscovici \cite{Connes82, CM90}.

If $X$ is smooth and has finite volume, but is not necessarily compact,  then Moscovici \cite{Moscovici82} constructed an index of elliptic operators on $X$. This was computed for Dirac operators by Barbasch and Moscovici \cite{BM83} in the case where the real rank of $G$ is $1$. (For groups of higher real rank, index theorems for the signature operator were obtained in \cite{Muller84, Stern89, Stern90}.) Barbasch and Moscovici's computation was based on a version of Selberg's trace formula for real rank  $1$ groups due to Osborne and Warner \cite{OW78, Warner79}. The result is an expression for the index of Dirac operators as a sum of various contributions, all but finitely many  of which are given by semisimple orbital integrals. In fact, if $X$ is compact, then one only needs these semisimple contributions, so the non-semisimple contributions reflect the  noncompactness of $X$. Some of these non-semisimple contributions are defined by conjugation-invariant functionals, which can be expected to define traces on suitable convolution algebras. The sum of the other contributions is conjugation-invariant as well, even though the individual terms are not.

Our goal here is to unify  these constructions and computations, and obtain a higher index theorem on finite-volume locally symmetric spaces if $G$ has real rank $1$. More precisely, we aim to construct a $K$-theoretic index from which
\begin{itemize}
\item the classical index  can be recovered, 
 \item  if $G$ has real rank $1$, then the conjugation-invariant parts of  Barbasch and Moscovici's expression for the classical index on $X$ can be extracted and computed individually, and
 \item the usual $\Gamma$-equivariant index  can be recovered if $X$ is compact.
 \end{itemize}
To achieve this aim, we show that
\begin{enumerate}
\item every $G$-equivariant, self-adjoint, elliptic, first-order differential operator on $G/K$ has a well-defined index in $K_*(\cA(G))$, where $\cA(G)$ is an analogue of Casselman's Schwartz algebra \cite{Casselman89};
\item this index can be represented by a standard idempotent involving heat operators \cite{CM90};
\item the semisimple orbital integrals and the conjugation-invariant  non-semisimple components of Barbasch and Moscovici's index theorem define continuous traces on $\cA(G)$. 
\end{enumerate}
There are several commonly used convolution algebras on $G$ besides Casselmann's Schwartz algebra on which semisimple orbital integrals define continuous traces. These include  $C^{\infty}_c(G)$ and Harish-Chandra's $L^1$- or $L^2$-Schwartz algebras. We use Casselman's algebra because it is large enough so that  the first two points above are true (the second is not true for $C^{\infty}_c(G)$), and small enough so that the third point is true (which is not true for some non-semisimple contributions if one uses Harish-Chandra's Schwartz algebras). 
See also Remark \ref{rem why SG}.

Because of the second point above, applying the traces from the third point to the index in the first point leads to expressions involving heat operators, which were  computed  in  \cite{BM83}. 
For traces not defined by semisimple orbital integrals, we then simply apply the computations in \cite{BM83}. For semisimple orbital integrals, we give an independent index-theoretic computation. This may be of independent interest, and it also has the advantages that 
\begin{itemize}
\item
it generalises to the higher orbital integrals from \cite{ST19}, which yield nonzero results even when $\rank(G)>\rank(K)$, in which case the classical index 
vanishes;
\item our computations also apply when $\Gamma$ has torsion, so that $X$ is an orbifold, and  we obtain nonzero contributions from nontrivial elliptic elements;
\item the index theorems for semisimple orbital integrals apply even when $G$ does not have real rank $1$. 
\end{itemize}
(We note, however, that the computations for semisimple orbital integrals are more about the action by $G$ on $G/K$ than about the group $\Gamma$.)

Our higher index on $G/K$ is constructed as a special case of a general higher index of elliptic operators equivariant with respect to proper, cocompact group actions, constructed in \cite{GHW25a}. This takes values in the $K$-theory of an algebra of exponentially fast decaying kernels. 


\subsection*{Acknowledgements}

We thank Yanli Song for helpful suggestions.

PH was supported by the Netherlands Organisation for Scientific Research NWO, through ENW-M grants  OCENW.M.21.176 and OCENW.M.23.063. HW was supported by the grants 23JC1401900, NSFC 12271165 and in part by Science and Technology Commission of Shanghai Municipality (No.\ 22DZ2229014).

\section{A higher index}


We recall  the construction of a higher index in the $K$-theory of an algebra of very rapidly decaying kernels from \cite{GHW25a}.

\subsection{An algebra of exponentially decaying kernels}\label{sec prelim alg}

Consider a  Riemannian manifold $X$, with  Riemannian distance $d$ and Riemannian density $\mu$.
We assume  that $X$ has \emph{uniformly exponential growth}, in the sense that there are $a_0,C>0$ such that  for all $x \in X$, 
\beq{eq unif exp}
\int_X e^{-{a_0} d(x,x')}\, d\mu(x') \leq C.
\eeq

%


We consider a Hermitian vector bundle $E \to X$. A section $\kappa$ of  the exterior tensor product $E \boxtimes E^* \to X \times X$ has \emph{finite propagation} if there is an $r>0$ such that for all $x,x' \in X$ with $d(x,x')>r$, we have $\kappa(x,x')=0$. If two $L^{\infty}$-sections $\kappa, \lambda$ of $E \boxtimes E^*$ have finite propagation, then their composition $\kappa \lambda$ is well-defined as an $L^{\infty}$-section of $E \boxtimes E^*$ with finite propagation, by
\beq{eq compos kernels}
(\kappa \lambda)(x,x') = \int_X \kappa(x,x'')\lambda(x'', x')\, d\mu(x''),
\eeq
for all $x, x' \in X$.

%
%

%
%

Let $H$ be a Lie group (possibly with infinitely many connected components, so $H$ may be discrete), acting isometrically on $X$ and such that $X/H$ is compact. Suppose that $E$ is an $H$-equivariant vector bundle, and that the action preserves the Hermitian metric on $E$. 
Consider the  action by $H \times H$ on $\Gamma^{\infty}(E \boxtimes E^*)$ given by
\beq{eq HxH action}
((h,h') \cdot \kappa)(x,x') := 
h\kappa(h^{-1}x, h'^{-1}x')h'^{-1}, 
\eeq
for $h,h' \in H$, $x,x' \in X$ and $\kappa \in \Gamma^{\infty}(E \boxtimes E^*)$. Suppose that the operators in $S$ commute with this action. 
%

Let $Y$ be another Riemannian manifold with uniformly exponential growth. Suppose that $H$ acts isometrically on $Y$. Let $\varphi\colon Y \to X$ be a surjective, $H$-equivariant smooth map.
We assume that $\varphi$ is
{quasi-isometry} in the sense that there are $A, B>0$ such that for all $y,y' \in Y$,
\beq{eq phi quasi isom}
A^{-1} d_Y(y,y') - B \leq
d_X(\varphi(y), \varphi(y')) \leq A d_Y(y,y') + B, 
\eeq
where $d_X$ is the Riemannian distance on $X$, and $d_Y$ is the one on $Y$. We also assume that
$\varphi$ is an {algebra homomorphism} in the sense that for all $\kappa, \lambda \in \Gamma_{\fp}^{\infty}(E \boxtimes E^*)$, 
\beq{eq phi homom}
\varphi^*(\kappa \lambda) =  (\varphi^* \kappa) (\varphi^*\lambda).
\eeq
In \eqref{eq phi homom}, we note that \eqref{eq phi quasi isom} implies that $ (\varphi^* \kappa)$ and $(\varphi^*\lambda)$ have finite propagation, so their composition is defined.

 Fix an algebra $S$ of differential operators on $\varphi^*E \boxtimes \varphi^*E^*$ such that
\begin{enumerate}
\item $S$ contains
the identity operator;
\item
for all $l \in \Z_{\geq 0}$, the set of order $l$ differential operators in $S$ is a finite-dimensional linear subspace of $S$;
\item for all $P \in S$, there are differential operators $Q$ and $R$ on $\varphi^*E$ such that $Q \otimes 1$ and $1\otimes R^*$ lie in  $S$ and $P = Q \otimes R^*$. Here $R^*$ is  the differential operator on $\varphi^*E^*$  dual to $R$.
%
 \end{enumerate}
Let $\Gamma_{S, \varphi, \fp}^{\infty}(E \boxtimes E^*)$ be the space of  smooth sections $\kappa$ of $E \boxtimes E^* \to X\times X$ with finite propagation, such that $P \varphi^*\kappa$ is bounded for all $P \in S$. 
\begin{definition}\label{def AX phi}
For $a>0$ and $P \in S$, the seminorm $\|\cdot \|_{a, P}$ on $\Gamma_{S, \varphi, \fp}^{\infty}(E \boxtimes E^*)$ is given by
\beq{eq seminorms ASE phi}
\|\kappa\|_{a, P}  = \sup_{y,y' \in Y} e^{ad_Y(y,y')}  \|(P \varphi^*\kappa) (y,y')\|, 
\eeq
for $\kappa \in \Gamma_{S, \fp}^{\infty}(E \boxtimes E^*)$, where $d_Y$ is the Riemannian distance on $Y$.
We denote the completion of $\Gamma_{S, \varphi, \fp}^{\infty}(E \boxtimes E^*)$  in the seminorms $\|\cdot \|_{a, P}$, for $a >0$ and $P \in S$, by $\cA_{S, \varphi}(E)$. The linear subspace of elements invariant under the action \eqref{eq HxH action} is denoted by $\cA_{S, \varphi}(E)^H$.
\end{definition}
\begin{lemma}
The space $\cA_{S, \varphi}(E)^H$ is a Fr\'echet algebra with respect to composition of kernels.
\end{lemma}
This lemma is a combination of Lemma 2.6 and equality (2.7) in \cite{GHW25a}.

\subsection{Construction of the index}

Let $D$ be an $H$-equivariant, first-order, elliptic, self-adjoint differential operator on $E$. We make the following assumptions.
\begin{enumerate}
\item There are $Q,R \in S$ such that   $\varphi^* \circ (D^2\otimes 1) = Q \circ \varphi^*$ and $\varphi^* \circ (1 \otimes (D^2)^*) = R \circ \varphi^*$.
\item  
For all $f \in C^{\infty}_c(X)$, all repeated commutators of $f$ with $D^2$ 
are well-defined, bounded operators on $L^2(E)$. Furthermore, if $Z$ is any smooth manifold, and $\tilde f \in C^{\infty}(Z \times X)$ is such that $\tilde f(z, \relbar) \in C^{\infty}_c(X)$ for all $z \in Z$, then we assume that the operator norms of the repeated commutators of $D^2$ with $\tilde f(z, \relbar)$ 
 depend continuously on $z$.  
\item For all $P \in S$ the compositions $P \circ \varphi^* \circ (D \otimes 1)$ and $P \circ \varphi^* \circ (1 \otimes D)$ are finite sums of compositions of the form $A \circ Q \circ \varphi^*$, for 
bounded endomorphisms $A$ of $\varphi^*E \boxtimes \varphi^* E^*$ and  $Q \in S$.
\end{enumerate}


Theorem 2.13 in \cite{GHW25a} is the following result.
\begin{theorem}\label{thm def index}
The operator $D$ lies in the multiplier algebra $\cM(\cA_{S, \varphi}(E)^H)$. Its class in $\cM(\cA_{S, \varphi}(E)^H)/\cA_{S, \varphi}(E)^H$ is invertible. The image of the resulting class  $ [D] \in K_1(\cM(\cA_{S, \varphi}(E)^H)/\cA_{S, \varphi}(E)^H)$ under the boundary map
\[
\partial\colon K_1(\cM(\cA_{S, \varphi}(E)^H)/\cA_{S, \varphi}(E)^H) \to K_0(\cA_{S, \varphi}(E)^H).
\]
 equals
\begin{multline} \label{eq index idempotent}
\partial [D] = \\
\left[
\begin{pmatrix}
e^{-tD^-D^+} & e^{-\frac{t}{2}D^-D^+} \frac{1-e^{-tD^-D^+}}{D^-D^+}D^-\\
e^{-\frac{t}{2}D^+D^-}D^+ & 1-e^{-tD^+D^-} 
\end{pmatrix}\right] - 
\left[
\begin{pmatrix}
0& 0\\
0 & 1
\end{pmatrix}\right]
\quad \in K_0(\cA_{S, \varphi}(E)^H).
\end{multline}
\end{theorem}

\begin{definition}\label{def H index}
The $H$-equivariant \emph{higher index} of $D$ with respect to $S$ and $\varphi$ is the class \eqref{eq index idempotent}. It is denoted by $\ind_{S,H}(D) \in K_0(\cA_{S, \varphi}(E)^H)$.
\end{definition}


\subsection{Special case: Riemannian symmetric spaces}\label{sec GK ex}


In this paper, we restrict ourselves to the motivating case of  Examples 2.3, 2.10 and 2.11 in \cite{GHW25a}. (See also Lemma \ref{lem AEG}.) We summarize these examples here, and refer to Examples 2.3, 2.10 and 2.11 in \cite{GHW25a} for the (short and straightforward) proofs of the claims made. 

Let $G$ be a noncompact,  connected, real semisimple Lie group, and let $K<G$ be a maximal compact subgroup.  Let $\kg = \kk \oplus \kp$ be a Cartan decomposition, and consider the 
$G$-invariant metric on $X=G/K$ defined by the $\Ad(K)$-invariant inner product on $\kg$ defined by the Killing form and the Cartan involution.

Let $V$ be a finite-dimensional unitary representation of $K$.
Consider the trivial vector bundle $G \times V \to G$.    Let $E = G \times_K V \to X$ be  the quotient of $G \times V$ by the $K$-action given by
\[
k\cdot (g,v) = (gk^{-1}, k \cdot v),
\]
for $k \in K$, $g \in G$ and $v \in V$.  Let $H<G$ be a closed subgroup such that $H \backslash G$ is compact, acting on the given spaces and vector bundles via left multiplication.
Let $q \colon G \to X$ be the quotient map. Then $q^*E = G \times V$, and $q$ satisfies all the assumptions on the map $\varphi$ made in Subsection \ref{sec prelim alg}. 

Let
\[
S:= \{R(P); P \in \cU(\kg \times \kg)\},
\]
where $R$ is the right regular representation of  $ \cU(\kg \times \kg)$ in $C^{\infty}(G \times G)$.
Then $S$ satisfies the three conditions listed above Definition \ref{def AX phi}. 

The adjoint representation $\Ad\colon K \to \SO(\kp)$ has a lift
\beq{eq tilde Ad}
\widetilde{\Ad}\colon 
\widetilde{K} \to \Spin(\kp)
\eeq
 to  double covers. 
In this way, the standard representation $\Delta_{\kp}$ of $\Spin(\kp)$ may be viewed as a representation of $\widetilde{K}$. 
 Let $W$ be a finite-dimensional, unitary representation of $\widetilde{K}$ such that $V = \Delta_{\kp} \otimes W$ descends to a representation of $K$.
 
 Let $\{X_1, \ldots, X_r\}$ be an orthonormal basis of $\kp$. Let $c\colon \kp \to \End(\Delta_{\kp})$ be the Clifford action. Then we have the Dirac operator
 \beq{eq DW}
 D_W := \sum_{j=1}^r R(X_j) \otimes c(X_j) \otimes 1_W
 \eeq
 on $\bigl( C^{\infty}(G) \otimes \Delta_{\kp} \otimes W\bigr)^K \cong \Gamma^{\infty}(G \times_K V)$. This operator satisfies the assumptions on $D$ listed above Theorem \ref{thm def index}. (This operator is used in many places, e.g. \cite{Atiyah77, BM83, Parthasarathy72, Wassermann87}.)

Suppose that $G/K$ is even-dimensional. Then $\Delta_{\kp}$, and hence $E$, has  a natural $\Z/2\Z$-grading, with respect to which $D_W$ is odd. 
By Theorem \ref{thm def index}, we obtain
\beq{eq ind H DW}
\ind_{S, H}(D_W) \in K_0(\cA_{S, q}(E)^H).
\eeq
In this paper, we study this index in the case where $H=G$.

\section{Index theorems}\label{sec index thms}

We continue using the notation and assumptions from Subsection \ref{sec GK ex}.
(Later, $G$ will sometimes be assumed to have real rank $1$.)
Fix a Haar measure $dg$ on $G$. 
Let  $dk$ be the Haar measure on $K$ giving it unit volume. The Riemannian distance on $X$ is denoted by $d$.  Let $d_{G}$ be the Riemannian distance on $G$ defined by the given inner product on $\kg$.

\subsection{Geometric traces}\label{sec prelim traces}

We will use some `geometric' traces on an algebra of matrix-valued functions on  $G$, defined by integrals over subsets of $G$. We also use some `spectral' traces, involving the operator trace, introduced in Subsection \ref{sec spec tr}. All these traces correspond to terms in Barbasch and Moscovici's index theorem \cite{BM83}.

We write $L$ and $R$ for the left and right regular representations of $G$ on smooth functions, respectively. We use the same notation for the corresponding representation  $\cU(\kg)$.
\begin{definition}
For $f \in C^{\infty}_c(G)$, $a>0$ and $P_1, P_2 \in \cU(\kg)$, we write 
\beq{eq norms SG}
\|f\|_{a,P_1, P_2} := \sup_{g \in G} e^{ad_G(e,g)} |(L(P_1)R(P_2)f)(g)|.
\eeq
 Let $\calS(G)$ be the completion of $C^{\infty}_c(G)$ in the seminorms $\| \cdot \|_{a,P_1, P_2}$.
 \end{definition}
\begin{remark}
The algebra $\calS(G)$ is a version of Casselman's Schwartz algebra \cite{Casselman89}. We use seminorms defined by left and right derivatives, whereas only left derivatives are used in  \cite{Casselman89}. We use both left and right derivatives so that $\calS(G)$ embeds into Harish-Chandra's Schwartz spaces, see Lemma \ref{lem incl SC}.  It is not essential to use both left and right derivatives, because in the end we will always restrict to subalgebras of $\calS(G)$ corresponding to a finite number of $K$-types. Then it is sufficient to only use left or right derivatives to obtain an embedding into Harish-Chandra's  Schwartz spaces, by (a slight extension of) Lemma 18 in \cite{HCDSII}. 
%
\end{remark}

\begin{lemma}\label{lem alg A infty}
 The space $\calS(G)$ is a Fr\'echet algebra with respect to convolution. 
 \end{lemma}
 This lemma is proved in Subsection \ref{sec Schwartz spaces}.
 
 A relation between the algebra $\calS(G)$ and the algebra from Definition \ref{def AX phi} will play a central role.  As in Subsection \ref{sec GK ex},
 we have the Fr\'echet algebra $\cA_{S, q}(E)^G$.
%
The Fr\'echet algebra $\calS(G) \otimes \End(V)$ is the completion of $C^{\infty}_c(G) \otimes \End(V)$ in the seminorms given by
\[
\|f\|_{a,P_1, P_2}^V := \sup_{g \in G} e^{ad_G(e, g)} \| (L(P_1)R(P_2)f)(g)\|_{\End(V)},
\]
for $a>0$, $P_1, P_2 \in \cU(\kg)$ and $f \in C^{\infty}_c(G) \otimes \End(V)$.
Consider the action by $K \times K$ on $\calS(G) \otimes \End(V)$ given by
\beq{eq action KK}
((k,k')\cdot f)(g) = \pi(k)f(k^{-1}gk')\pi(k')^{-1},
\eeq
for $k,k' \in K$, $f \in \calS(G) \otimes \End(V)$ and $g \in G$. Analogously to Lemma 2.6 in \cite{GHW25a}, elements of $K\times K$ act on $\calS(G) \otimes \End(V)$ by continuous operators, so that the fixed point set $(\calS(G) \otimes \End(V))^{K \times K}$ is a Fr\'echet algebra. 

Consider the fixed-point set $(C^{\infty}(G) \otimes \End(V))^{K \times K}$ of the action \eqref{eq action KK} by $K \times K$  on $C^{\infty}(G) \otimes \End(V)$.
\begin{lemma}\label{lem AEG}
For all $\kappa \in \Gamma^{\infty}(E \boxtimes E^*)^G$, setting 
\beq{eq def Psi}
[e, \Psi(\kappa)(g)v] = \kappa(eK,gK)[g,v],
\eeq
for $g \in G$ and $v \in V$, defines a function $\Psi(\kappa) \in (C^{\infty}(G) \otimes \End(V))^{K \times K}$.
This defines a linear isomorphism $\Psi\colon \Gamma^{\infty}(E \boxtimes E^*)^G \to (C^{\infty}(G) \otimes \End(V))^{K \times K}$.  This linear isomorphism restricts to an algebra isomorphism from 
$\Gamma^{\infty}_{\fp}(E \boxtimes E^*)^G$ to $(C^{\infty}_c(G) \otimes \End(V))^{K \times K}$, 
which in turn extends continuously to an isometric algebra isomorphism 
 from $\cA_{S, q}(E)^G$ to $(\calS(G) \otimes \End(V))^{K \times K}$.
\end{lemma}
This lemma is proved  in Subsection \ref{sec Schwartz spaces}.

We will use several continuous traces on the algebra $(\calS(G) \otimes \End(V))^{K \times K}$.
If $\tau$ is a linear functional on $\calS(G)$, or on the subspace of $\calS(G)$ corresponding to the $K\times K$-isotypical components in $\End(V)$, then we write
\[
\tau^V := \tau \otimes \tr_V \colon (\calS(G) \otimes \End(V))^{K \times K} \to \C.
\]

If $H$ is any group, $H'<H$ a subgroup, and $h \in H$, then we write $H'_h$ for the centraliser of $h$ in $H'$.

Let $g \in G$ be such that $G_g$ is unimodular. Fix a Haar measure $dh$ on $G_g$. Because $G_g$ is unimodular, there is a unique $G$-invariant measure $d(G_g h)$ on $G_g \backslash G$ such that for all $f \in C_c(G)$,
\[
\int_G f(h)\, dh = \int_{G_g \backslash G} \int_{G_g} f(h'h)\, dh'\, d(G_g h).
\]
For $f \in C^{\infty}(G)$ such that the following converges, consider the {orbital integral}
\beq{eq taug}
\tau_g(f) := \int_{G_g \backslash G} f(h^{-1}gh)\, d(G_g h).
\eeq
\begin{proposition}\label{prop tr SG}
If $g$ is semisimple, then $\tau_g^V$ defines a continuous trace on $(\calS(G) \otimes \End(V))^{K \times K}$.
\end{proposition}
This proposition is proved in Subsection \ref{sec traces}.

In Theorem \ref{thm index ss}, we will see that semisimple orbital integrals (the traces from Proposition \ref{prop tr SG}) map the higher index of the operators we consider to zero if $\rank(G)> \rank(K)$. The higher index itself is generally nonzero in that case, however, and this nonvanishing can be detected by pairing the index with cyclic cocycles defined by the \emph{higher orbital integrals} from \cite{ST19}.


Let $P<G$ be a cuspidal parabolic subgroup, with Langlands decomposition $P = MA'N'$. (We use primes to distinguish the groups $A'$ and $N'$ from the groups $A$ and $N$ in an Iwasawa decomposition $G = KAN$ below.) Let $g \in M$ be semisimple. In Definition 3.3  and Theorem 3.5 in \cite{ST19}, Song and Tang constructed a cyclic cocycle
$
\Phi_{P, g} 
$
over Harish-Chandra's Schwartz algebra $\cC^2(G)$ (see page \pageref{page Cp}), whose degree is the dimension of $A'$. If $\rank(G) = \rank(K)$, then $G$ is a cuspidal parabolic subgroup of itself, and $\Phi_{G, g} = \tau_g$.

\begin{proposition} \label{prop cocycle SG}
The operator trace on $\End(V)$ and the cyclic cocycle $\Phi_{P, g}$ combine into a cyclic cocycle $\Phi_{P, g}^V$ on $(\calS(G) \otimes \End(V))^{K \times K}$.
\end{proposition}
This proposition is proved in Subsection \ref{sec traces}.

We will also use some traces that are not associated to semisimple elements of $G$. From now on, until the end of Subsection \ref{sec index base}, we suppose that $G$ has real rank $1$.
%
%
Consider an Iwasawa decomposition $G = KAN$, and a compatible $\Gamma$-cuspidal parabolic subgroup $P= MAN$. Then $A$ is one-dimensional by assumption. Let $\lambda \in \ka^*$ be the simple restricted root of $(\kg, \ka)$ that determines $N$. Then we have the restricted root space decomposition $\kn = \kn_{\lambda} \oplus \kn_{2\lambda}$. With $N_{\lambda}:= \exp(\kn_{\lambda})$ and $N_{2\lambda}:= \exp(\kn_{2\lambda})$, we have $N = N_{\lambda} N_{2\lambda}$, where $N_{\lambda} \cap N_{2\lambda} = \{e\}$.

 Since $N_{2\lambda}$ lies in the centre of $N$, we have $G_{n} = M_{n} N$ for all nontrivial $n \in N_{2\lambda}$, so $G_n$ is unimodular. So the orbital integral $\tau_n(f)$ can be defined on suitable functions.
\begin{lemma}\label{lem T2lam}
For all $n \in N_{2\lambda}$, the orbital integral $\tau_n^V$ defines a continuous trace on $(\calS(G) \otimes \End(V))^{K \times K}$.
\end{lemma}
This lemma is proved in Subsection \ref{sec traces}.

For  $a \in A$ and $\xi \in \ka^*$, we write $a^{\xi} := e^{\langle \xi, \log a\rangle}$. Let 
\beq{eq def rho}
\rho_N:= (\dim(\kn_{\lambda})/2+\dim(\kn_{2\lambda}))\lambda.
\eeq
In a few places, we will use the fact that the Haar measure $dg$ on $G$ can be decomposed  as
\beq{eq Haar KAN}
\int_{G} f(g)\, dg = \int_{KAN} f(kan) a^{2\rho_N}\, dk\, da\, dn,
\eeq
for all $f \in C_c(G)$, 
for Haar measures $dk$ on $K$, $da$ on $A$ and $dn$ on $N$.

For $f \in C^{\infty}(G)$ such that the following converges, write
\beq{eq def Tlambda}
\tau_{\lambda}(f) := \int_N \int_K f(knk^{-1})\, dk\, dn.
\eeq
\begin{proposition}\label{prop Tlam}
The integral \eqref{eq def Tlambda} defines a continuous trace $\tau_{\lambda}^V$ on $(\calS(G) \otimes \End(V))^{K \times K}$.
\end{proposition}
This proposition is proved in Subsection \ref{sec traces}.


\subsection{Spectral traces}\label{sec spec tr}

Apart from the geometric traces of Subsection \ref{sec prelim traces}, we also use some `spectral' traces on $(\calS(G) \otimes \End(V))^{K \times K}$.

Let $\Gamma<G$ be a discrete subgroup such that $\Gamma \backslash G$ has finite volume. 
As in \cite{BM83, OW78, Warner79}, we make the following simplifying technical assumption. Let $P_1, \ldots, P_l$ be representatives of the $\Gamma$-conjugacy classes of  $\Gamma$-cuspidal parabolic subgroups of $G$, with Langlands decompositions $P_j = M_j A_j N_j$. Let $Z_{\Gamma}$ be the centre of $\Gamma$. Then we assume that for all $j$, 
\beq{eq condition Gamma}
\Gamma \cap P_j = Z_{\Gamma} \cdot (\Gamma \cap N_j).
\eeq
This condition implies that a noncentral element $\gamma \in \Gamma$ is semisimple if and only if its conjugacy class has empty intersection with each of the groups $P_j$; see Lemma 5.3 in \cite{Warner79}. That in turn implies that for all semisimple $\gamma \in \Gamma$, the quotient $\Gamma_{\gamma} \backslash G_{\gamma}$ has finite volume; see Lemma 5.4 in \cite{Warner79}. In general, a discrete subgroup $\Gamma<G$ such that $\Gamma \backslash G$ has finite volume has a finite-index normal subgroup with the above property. (See page 56 of \cite{Warner79} and Lemma 6.5 in \cite{GR70}.) It is possible to work without this condition; then noncentral semisimple elements of $\Gamma$ should be replaced by noncentral elements whose  conjugacy classes have empty intersections with the groups $P_j$.
\begin{example}
Let $G = \SL(2, \R)$ and $\Gamma = \SL(2, \Z)$. There is one $\Gamma$-conjugacy class of $\Gamma$-cuspidal parabolic subgroups, represented by the subgroup $P<G$ of upper-triangular matrices. In its Langlands decomposition $P = MA'N'$, we have
\[
\begin{split}
N' &= \left\{  
\begin{pmatrix}
1 & b \\ 0 & 1
\end{pmatrix}; b \in \R
\right\}, \quad \text{so}\\
\Gamma \cap N' &= \left\{  
\begin{pmatrix}
1 & b \\ 0 & 1
\end{pmatrix}; b \in \Z
\right\}.
\end{split}
\]
And $Z_{\Gamma} = \{\pm I\}$, so 
\[
Z_{\Gamma} \cdot (\Gamma \cap N') = \left\{  
\begin{pmatrix}
a & b \\ 0 & a
\end{pmatrix}; a \in \{\pm 1\}, b \in \Z
\right\} = \Gamma \cap P.
\]
So \eqref{eq condition Gamma} holds.
\end{example}

There are $k_1, \ldots, k_l \in K$ such that $P_j = k_j P k_j^{-1}$ for all $j$. So if $\sigma \in \hat M$, then  via conjugation, we have the representation $k_j \cdot \sigma$ of $M_j = k_j M k_j^{-1}$. This yields the representation
\[
\pi_{\Gamma}(\sigma) :=   \dim(\sigma) \Ind_{P_j}^G(k_j \cdot \sigma \otimes 1_{AN})
\]
of $G$. There is a unitary intertwining operator $c_{\Gamma}(\sigma)$ for this representation that squares to the identity.
\begin{lemma}\label{lem tau res}
The map 
\beq{eq tau res}
f\mapsto \sum_{\sigma \in \hat M} (\Tr \otimes \tr_V)(c_{\Gamma}(\sigma) \pi_{\Gamma}(\sigma)(f))
\eeq
defines a continuous trace $\tau_{\res}$ on $(\calS(G) \otimes \End(V))^{K \times K}$.
\end{lemma}
This lemma is proved in Subsection \ref{sec trace Gamma}.

Consider the right quasiregular representation $R^{\Gamma}$ of $G$ in $L^2(\Gamma \backslash G)$: for all $g,h \in G$ and $\varphi \in L^2(\Gamma \backslash G)$,
\[
(R^{\Gamma}(g)\varphi)(\Gamma h) = \varphi(\Gamma hg).
\]
Write $L^2(\Gamma \backslash G) = L^2_d(\Gamma \backslash G) \oplus L^2_c(\Gamma \backslash G)$, where $L^2_d(\Gamma \backslash G)$ decomposes discretely into irreducible representations of $G$, and $L^2_c(\Gamma \backslash G)$ decomposes continuously. Let $R^{\Gamma}_d$ be the restriction of   $R^{\Gamma}$ to $L^2_d(\Gamma \backslash G)$. It defines a representation of $L^1(G)$ in the usual way. We consider its extension
\[
(R^{\Gamma}_d)^V\colon (L^1(G) \otimes \End(V))^{K \times K} \to \cB((L^2_d(\Gamma \backslash G) \otimes V)^K).
\]

\begin{proposition}\label{prop Rd cts}
For all $f \in (\calS(G)\otimes \End(V))^{K \times K} $, the operator $(R_d^{\Gamma})^V(f)$ is trace-class. 
The algebra homomorphism
\beq{eq RdV}
(R_d^{\Gamma})^V\colon(\calS(G)\otimes \End(V))^{K \times K} \to \calL^1((L^2_d(\Gamma \backslash G) \otimes V)^K) 
\eeq
is continuous.
\end{proposition}
\begin{corollary}\label{cor Rd trace}
The map $f \mapsto \Tr(R_d^{\Gamma}(f))$ defines a continuous trace $(\Tr \circ R_d^{\Gamma})^V$ on $(\calS(G) \otimes \End(V))^{K \times K}$.
\end{corollary}
Proposition \ref{prop Rd cts} and Corollary \ref{cor Rd trace} are proved in Subsection \ref{sec trace Gamma}.

Let $A(\kn_{\lambda})$ be the area of the unit sphere in $\kn_{\lambda}$, and let $C_{\lambda}(\Gamma)$ and $C_{2\lambda}(\Gamma)$ be as on page 299 of \cite{OW78}. Fix $n_0 \in N_{2\lambda}$ such that $\|\log(n_0)\| = 1$.
For $f \in C^{\infty}(G)$ such that  $R^{\Gamma}_d(f)$ is trace-class and the other terms converge, we consider the `remainder' 
\begin{multline} \label{eq def T}
\tau_{\rem}(f) := \Tr(R^{\Gamma}_d(f)) + \frac{1}{4} \tau_{\res} (f)\\
- \sum_{(\gamma) \text{ semisimple}} \vol(\Gamma_{\gamma}\backslash G_{\gamma})\tau_{\gamma}(f) -  \frac{C_{2\lambda}(\Gamma)}{2}(\tau_{n_0}(f) + \tau_{n_0^{-1}}(f)) -  \frac{C_{\lambda}(\Gamma)}{A(\kn_{\lambda}) |\lambda|} \tau_{\lambda} (f),
\end{multline}
Here the sum on the right hand side is over the $\Gamma$-conjugacy classes $(\gamma)$ of semisimple elements of $\Gamma$.

\begin{remark}
An explicit form of the functional $\tau_{\rem}$ is given in \cite{OW78}: see the theorems on pages 293 and 299 in that paper.
\end{remark}

\begin{proposition}\label{prop trace T}
The functional $\tau_{\rem}$ defines a continuous trace $\tau_{\rem}^V$  on $(\calS(G) \otimes \End(V))^{K \times K}$.
\end{proposition}
This proposition is proved at the end of Subsection \ref{sec trace Gamma}.

\subsection{Relation with the classical index and the $\Gamma$-index}\label{sec index base}

We first note that the index of elliptic operators on real rank one, finite-volume locally symmetric spaces studied in \cite{BM83, Moscovici82} can be recovered from the higher index of Definition \ref{def H index}. 




Consider the setting of Subsection \ref{sec GK ex}. 
We consider the case of the index \eqref{eq ind H DW} where $H=G$:
\beq{eq ind G DW}
\ind_{S, G}(D_W) \in K_0(\cA_{S, q}(E)^G).
\eeq
Using the isomorphism $\Psi\colon \cA_{S, q}(E)^G \to (\calS(G) \otimes \End(V))^{K \times K}$ from Lemma \ref{lem AEG}, we can map this index to $K_0((\calS(G) \otimes \End(V))^{K \times K})$. Then applying the map induced by the algebra homomorphism   $(R^{\Gamma}_d)^V$ from Proposition \ref{prop Rd cts} yields an element of 
\beq{eq iso L1}
K_0\bigl(
\calL^1(L^2_d(\Gamma \backslash G) \otimes V)\bigr) = \Z.
\eeq

Let $D_W^{\Gamma}$ be the elliptic differential operator on $\Gamma \backslash G \times_K V \to \Gamma \backslash G/K$ induced by $D_W$. It has a well-defined index by Theorem 2.1 in \cite{Moscovici82}. 
\begin{corollary}\label{cor DGamma}
If $G$ has real rank $1$, then 
under the isomorphism \eqref{eq iso L1},
 we have
\beq{eq higher lower}
(R^{\Gamma}_d)^V _* \circ \Psi_*(\ind_{S, G}(D_W) ) = \ind(D_W^{\Gamma}).
\eeq
\end{corollary}
\begin{proof}
The isomorphism \eqref{eq iso L1} is given by the operator trace. So
by Theorem \ref{thm def index}, the left hand side of \eqref{eq higher lower} equals
\[
\Tr\bigl( R^{\Gamma}_d(\Psi(e^{-t D_W^-D_W^+}) ) \bigr) - \Tr\bigl( R^{\Gamma}_d(\Psi(e^{-t D_W^+D_W^-}) ) \bigr).
\]
By Proposition 3.1 in \cite{BM83}, this equals the right hand side of \eqref{eq higher lower}.
\end{proof}

\begin{remark}
The number \eqref{eq higher lower} is the value of the continuous trace $\Tr \circ (R_d^{\Gamma})^V \circ \Psi$ on $\ind_{S, G}(D_W)$.
\end{remark}

\begin{remark}
By (1.2.5) in \cite{BM83}, the index \eqref{eq higher lower} vanishes if $\rank(G) > \rank(K)$. The higher index \eqref{eq ind G DW} will generally be nonzero, however. Indeed, Corollary 4.4 in \cite{GHW25a} states that the higher index \eqref{eq ind G DW}  maps to the standard $G$-index $\ind_G(D_W) \in K_0(C^*_r(G))$ under the maps induced on $K$-theory by $\Psi$ and the inclusion $\calS(G) \hookrightarrow C^*_r(G)$.
Bradd \cite{Bradd23, BH21} showed that the inclusion $\calS(G) \hookrightarrow C^*_r(G)$ induces an isomorphism on $K$-theory, if one restricts to isotypical components corresponding to certain finite sets of $K$-types. 
Because $\ind_G(D_W)$ is nonzero by the Connes--Kasparov conjecture, so is 
 the higher index \eqref{eq ind G DW}. 
%
%
%
\end{remark}


If $X$ is \emph{compact} and $\Gamma$ is torsion-free, then the algebra $\cA_{S, q}(E)^{\Gamma}$ is continuously included in the equivariant Roe algebra $C^*(G/K)^{\Gamma}$. Then
\[
\ind_{S, \Gamma}(D_W) \in K_0(\cA_{S, q}(E)^{\Gamma})
\]
is mapped to the usual $\Gamma$-index
\[
\ind_{\Gamma}(D_W) \in K_0(C^*_r(\Gamma)) = K_0(C^*(G/K)^{\Gamma})
\]
by the map induced by the inclusion. 
See Remark 2.16 in \cite{GHW25a} and \cite{Roe02}.

\subsection{Index theorems:  semisimple and non-semisimple cases}\label{sec prelim ind thm}

Next, we use the algebra isomorphism from Lemma \ref{lem AEG} to apply the traces from Subsections \ref{sec prelim traces} and \ref{sec spec tr} to the higher index \eqref{eq ind G DW}, and obtain a higher version of the index theorem in \cite{BM83}. For now  we drop the assumption that $G$ has real rank $1$; this will be assumed again in the index theorem for non-semisimple contributions, Theorem \ref{thm index non-ss}.

The index theorems are stated for
 Dirac operators  \eqref{eq DW} on $G/K$.
Consider the traces $\tau_{\gamma}^{\Delta_{\kp} \otimes W}$ from Proposition \ref{prop tr SG}, for semisimple $\gamma \in \Gamma$,  $(\tau_{n_0} + \tau_{n_0^{-1}})^{\Delta_{\kp} \otimes W}$ from Lemma \ref{lem T2lam}, $\tau_{\lambda}^{\Delta_{\kp} \otimes W}$ from Proposition \ref{prop Tlam} and $T_{\lambda}^{\Delta_{\kp} \otimes W}$ from Proposition \ref{prop trace T} on the algebra $(\calS(G) \otimes \End(\Delta_{\kp} \otimes W))^{K \times K}$. Via the isomorphism $\Psi$ from Lemma \ref{lem AEG}, these define continuous traces on $\cA_{S, q}(E)^G$. These can be applied to \eqref{eq ind G DW}. In fact, higher cyclic cocycles $\Phi_{P, \gamma}^{\Delta_{\kp} \otimes W}$ as in Proposition \ref{prop cocycle SG}, for $\gamma$ semisimple, can also be applied to the index \eqref{eq ind G DW} and computed, yielding possibly nonzero results in cases where the trace $\tau_{\gamma}^{\Delta_{\kp} \otimes W}$ maps \eqref{eq ind G DW}  to zero.

 To state index theorems for the resulting numbers, we introduce some notation in the various cases.

\subsubsection*{Index theorem for semisimple contributions}

Let $T<K$ be a maximal torus. 
If $T$ is a Cartan subgroup of $G$, so $\rank(G) = \rank(K)$, then we  use the following notation.  
Fix a positive root system $R^+(G,T)$ for $(G, T)$, and consider the Weyl group $W_K = N_K(T)/T$. Let $\rho$ be half the sum of the roots in $R^+(G, T)$, 
 let $\rho_c$ be half the sum of the  roots in $R^+(K, T) = R^+(G, T) \cap R(K,T)$, and let $\rho_n = \rho-\rho_c$.
 
 If $g \in T$, then we write    $R^+(G_g,T) =  R^+(G,T) \cap R(G_g, T)$, and $W_{K_{g}} = N_{K_g}(T)/T$. Let $\rho_g$ be half the sum of the roots in $R^+(G_g, T)$.
 
 Let $\lambda$ be the highest weight of $W$ with respect to the set of compact positive roots $R^+(K, T) \subset R^+(G, T)$. For $\mu \in i\kt^*$, we write
\beq{eq def dG}
d^{G_g}_{\mu} = \prod_{\alpha \in R^+(G_g, T)}\frac{(\mu, \alpha)}{(\rho_g, \alpha)}. 
\eeq
(This includes the special case where $g=e$, so $G_g = G$.)
For a discrete series representation of $G_g$ with Harish-Chandra parameter $\mu$, this is its formal degree.

\begin{theorem}[Index theorem, semisimple case]\label{thm index ss}
Suppose  that $\rank(G) = \rank(K)$.
\begin{enumerate}
\item[(a)] If $\gamma \in \Gamma$ is central in $G$, then
\[
\tau_{\gamma}^{\Delta_{\kp} \otimes W} \circ \Psi_*(\ind_{S, G}(D_W) ) = e^{\lambda - \rho_n}(\gamma)d^G_{\lambda+\rho_c}.
\]
\item[(b)] If $\gamma \in \Gamma$ is hyperbolic, then
\[
\tau_{\gamma}^{\Delta_{\kp} \otimes W}\circ \Psi_*(\ind_{S, G}(D_W) ) = 0.
\]
\item[(c)] If $\gamma \in \Gamma$ is elliptic, then with respect to a compact Cartan subgroup $T$ containing $\gamma$, 
\begin{multline*}
\tau_{\gamma}^{\Delta_{\kp} \otimes W}\circ \Psi_*(\ind_{S, G}(D_W) ) = \\
(-1)^{(\dim(G/K) + \dim(G_{\gamma}/K_{\gamma})/2} 
\frac{\sum_{w \in W_{K_{\gamma}} \backslash W_K}  \det(w) e^{w(\lambda + \rho_c - \rho_{\gamma})}(\gamma) d^{G_{\gamma}}_{w(\lambda+ \rho_{\gamma}-\rho_n)}}{e^{\rho- \rho_{\gamma}}(\gamma)\prod_{\alpha \in R^+(G, T) \setminus R^+(G_{\gamma}, T)} (1-e^{\alpha}({\gamma})) }.
\end{multline*}
\end{enumerate}
\end{theorem}

\subsubsection*{Index theorem for semisimple contributions from higher cocycles}

As in  Proposition \ref{prop cocycle SG}, consider a cuspidal parabolic subgroup $P = MA'N'<G$, and a semisimple element $g \in M$.
If $P$ is a maximal cuspidal parabolic subgroup, then $T<M$,  after conjugation if necessary. Then $T$ is a compact Cartan subgroup of $M$.
 Fix a positive root system $R^+(M, T) \subset R(M, T)$.
If $g \in T$, we write
\[
R^+(M_g, T) := R(M_g,  T) \cap R^+(M, T).
\]
Let $\rho^M$ be half the sum of the elements of $R^+(M, T)$,  let $\rho^M_c$ be half the sum of the elements of $R^+(K \cap M, T) := R^+(M, T) \cap R(K \cap M, T)$, and let $\rho^M_n = \rho^M - \rho^M_c$. Let $\rho^M_g$ be half the sum of the roots in $R^+(M_g, T)$.

As in \cite{HST20}, we now assume that the map \eqref{eq tilde Ad}
 descends to $K$. By Lemma 4.3 in \cite{HST20}, there is a natural $K \cap M$-equivariant embedding of $\kk/(\kk \cap \km)$ into $\kp$, via an identification of  $\kk/(\kk \cap \km)$ with the orthogonal complement of $(\kp\cap \km) \oplus \ka$. Hence we obtain the representation $K\cap M \to \Spin(\kk/(\kk \cap \km)) \to \GL(\Delta_{\kk/(\kk \cap \km)})$. The spaces $\kk/\kt$ and $(\kk\cap \km)/\kt$ are both even-dimensional, hence so is $\kk/(\kk \cap \km)$. So the representation $\Delta_{\kk/(\kk \cap \km)}$ of $K \cap M$ has a natural $\Z/2\Z$ grading $\Delta_{\kk/(\kk \cap \km)} = \Delta_{\kk/(\kk \cap \km)}^+  \oplus \Delta_{\kk/(\kk \cap \km)}^-$. If we view  $\Delta_{\kk/(\kk \cap \km)}$ as the virtual representation $[\Delta_{\kk/(\kk \cap \km)}^+]  - [\Delta_{\kk/(\kk \cap \km)}^-] \in R(K \cap M)$, then some of the multiplicities $m_U \in \Z$ in
\beq{eq decomp Delta}
\Delta_{\kk/(\kk \cap \km) } \otimes W|_{K \cap M} = \bigoplus_{U \in \widehat{K \cap M}} m_U U,
\eeq
may be negative. 

For each $U \in \widehat{K \cap M}$, let $\lambda_U \in i(\kt \cap \km)^*$ be its highest weight with respect to the positive root system $R^+(K \cap M, T)$. Consider the Weyl groups $W_{K \cap M} = N_{K \cap M}(T)/T$ and $W_{K \cap M_g} = N_{K \cap M_g}(T)/T$. 
For $\mu \in i(\kt \cap \km)^*$, define $d^{M_g}_{\mu}$ analogously to \eqref{eq def dG}, where the product now runs over $\alpha \in R^+(M_g, T)$.
For semisimple $g \in M$,  let $\Phi^{\Delta_{\kp} \otimes W}_{P, g}$ be as in Proposition \ref{prop cocycle SG}. 
\begin{theorem}[Index theorem, semisimple case for higher cocycles]\label{thm index ss higher}
If $P$ is not a maximal cuspidal parabolic subgroup of $G$, then 
\[
\langle \Phi^{\Delta_{\kp} \otimes W}_{P, \gamma},
 \Psi_*(\ind_{S, G}(D_W) ) \rangle = 0
 \]
  for all semisimple $\gamma \in \Gamma \cap M$.
Suppose now that $P$ is  a maximal cuspidal parabolic subgroup of $G$.
\begin{enumerate}
\item[(a)] If $\gamma \in \Gamma \cap M$ is central in $M$, then
\[
\langle \Phi^{\Delta_{\kp} \otimes W}_{P, \gamma},
 \Psi_*(\ind_{S, G}(D_W) ) \rangle = 
 \sum_{U \in \widehat{K \cap M}} m_U e^{\lambda_U - \rho^M_n}(\gamma)
d^{M}_{\lambda_U+ \rho^M_c}.
\]
\item[(b)] If $\gamma \in \Gamma \cap M$ is hyperbolic in $M$, then
\[
\langle \Phi^{\Delta_{\kp} \otimes W}_{P, \gamma},
 \Psi_*(\ind_{S, G}(D_W) ) \rangle = 0.
\]
\item[(c)] If $\gamma \in \Gamma \cap M$ is elliptic in $M$, then 
with respect to a compact Cartan subgroup $T<M$ containing $\gamma$, 
\begin{multline*}
\langle \Phi^{\Delta_{\kp} \otimes W}_{P, \gamma},
 \Psi_*(\ind_{S, G}(D_W) ) \rangle = \\
 (-1)^{(\dim(M/(K \cap M)) + \dim(M_{\gamma}/(K \cap M_{\gamma}))/2} \cdot \\
\sum_{U \in \widehat{K \cap M}} m_U
\frac{\sum_{w \in W_{K \cap M_{\gamma}} \backslash W_{K \cap M}}  \det(w) e^{w(\lambda_U + \rho^M_c - \rho^M_{\gamma})}(\gamma) d^{M_{\gamma}}_{w(\lambda_U+ \rho^M_{\gamma} -\rho^M_n)}}{e^{\rho^M- \rho^M_{\gamma}}(\gamma)\prod_{\alpha \in R^+(M, T) \setminus R^+(M_{\gamma}, T)} (1-e^{\alpha}(\gamma)) }.
%
\end{multline*}
\end{enumerate}
\end{theorem}

\begin{remark}
If $\rank(G)=\rank(K)$, then $G$ is a maximal cuspidal parabolic subgroup of itself, and $\Phi_{G, g} = \tau_g$ for all semisimple $g \in G$. Hence Theorem \ref{thm index ss higher} reduces to Theorem \ref{thm index ss}  in this case.
\end{remark}

\subsubsection*{Index theorem for non-semisimple contributions}

To state an index theorem for traces not associated to semisimple elements of $\Gamma$, we introduce some further notation. We now assume that $G$ has real rank $1$.

Let $l$ be the number of $\Gamma$-conjugacy classes of $\Gamma$-cuspidal parabolic subgroups of $G$. Choose representatives $P_j = M_j A_j N_j$ of these conjugacy classes. If $G = \SU(2n, 1)$, then we write
\[
C_2(\Gamma) := \sum_{j=1}^{l} \frac{\vol(\Gamma \cap N_j \backslash N_j)}{\vol(\Gamma \cap (N_j)_{2\lambda} \backslash (N_j)_{2\lambda})} \frac{(2\pi)^{2n}}{(2n)!}B_n,
\]
where $B_n$ is the $n$th Bernoulli number. Furthermore, let $\kt \subset \ksu(2n,1)$ be the diagonal Cartan. Let  $Z_0 \in \kt$ be the diagonal matrix with diagonal elements $i, i, \ldots, i, -2ni$. Let
\[
\varepsilon(R^+(G, T)) = \sign \prod_{\alpha \in R^+(G, T) \setminus R^+(K, T)} \langle \alpha, Z_0\rangle.
\]

Let $\Omega \in \cU(\kg_{\C})$ be the Casimir element. For $\sigma \in \hat M$, consider the operators
\[
c_{\lambda}(\sigma) := - \Ind_P^G(\sigma \otimes 1_{AN}) (\Omega)+ \|\lambda + \rho_c\|^2 - \|\rho\|^2
\]
on the representation space $\calH(\sigma)$ of $\Ind_P^G(\sigma \otimes 1_{AN})$. Let
$
c^+_{\Gamma, \lambda}(\sigma) 
$
be the restriction of $c_{\Gamma}(\sigma) \otimes 1_{\Delta_{\kp}^+ \otimes W}$ to $K$-invariant vectors. (Here $c_{\Gamma}(\sigma)$ is as in Lemma \ref{lem tau res}.)

For general $G$, if $R^+_0(G, T)$ is a positive root system chosen as in Section 2.a of \cite{Arthur74}, then we write
\[
\varepsilon(\lambda + \rho_c) := \sign \prod_{\alpha \in R^+_0(G, T)} (\lambda + \rho_c, \alpha).
\]
If $N_{2\lambda}$ is trivial, then we choose a vector $\tilde X \in \kn_{\lambda}$ of norm $\frac{\sqrt{2}}{\|\lambda\|}$. If  $N_{2\lambda}$ is nontrivial, then we choose a vector $\tilde X \in \kn_{2\lambda}$ of norm $\frac{1}{\sqrt{2}\|\lambda\|}$. Write $\tilde Y := \theta(\tilde X)$, where $\theta$ is the Cartan involution. For integral $\mu \in i\kt^*$, define $k(\mu) \in \Z$ by the relation
\[
e^{\langle \mu, t(\tilde X - \tilde Y)\rangle} = e^{ik(\mu)t}
\]
for all $t \in \R$. Recall the choice of $n_0 \in N_{2\lambda}$ above \eqref{eq def T}. Consider the Weyl group $W_M = N_M(T \cap M)/(T\cap M)$.  
\begin{theorem}[Index theorem, non-semisimple case]\label{thm index non-ss}
Suppose that  $G$ has real rank $1$ and that $\rank(G) = \rank(K)$.
\begin{enumerate}
\item[(a)] If $G \not= \SU(2n,1)$, then
\[
(\tau_{n_0} + \tau_{n_0^{-1}})^{\Delta_{\kp} \otimes W}\circ \Psi_*(\ind_{S, G}(D_W) ) = 0.
\]
If $G = \SU(2n,1)$, then
\begin{multline}\label{eq ind thm taun}
(\tau_{n_0} + \tau_{n_0^{-1}})^{\Delta_{\kp} \otimes W}\circ \Psi_*(\ind_{S, G}(D_W) ) = \\
\|\lambda\|
\frac{2^{3n-3}(2n+1)^n}{\area(S^{4n-1})} C_2(\Gamma) \varepsilon(R^+(G, T)) \dim(W).
\end{multline}
\item[(b)] We have
\[
\tau_{\lambda}^{\Delta_{\kp} \otimes W}\circ \Psi_*(\ind_{S, G}(D_W) ) = 0.
\]
\item[(c)]
If $\lambda + \rho_c$ is regular, then
\[
\tau_{\res} \circ \Psi_*(\ind_{S, G}(D_W) ) = 0.
\]
If $\lambda + \rho_c$ is singular, then
\[
\tau_{\res}\circ \Psi_*(\ind_{S, G}(D_W) ) = 2\sum_{\sigma \in \hat M; c_{\lambda}(\sigma)=0} \Tr(c_{\Gamma, \lambda}^+(\sigma)).
\]
\item[(d)]
If $\lambda + \rho_c$ is regular, then
\begin{multline} \label{eq ind thm T}
\tau_{\rem}^{\Delta_{\kp} \otimes W}\circ \Psi_*(\ind_{S, G}(D_W) ) = \frac{(-1)^{\dim(G/K)/2}}{2} l \varepsilon(\lambda + \rho_c)  \\
\sum_{w \in W_{M} \setminus W_K}\det(w) \sgn( k(w(\lambda+\rho_c))) 
\dim(\sigma_{w_M w (\lambda + \rho_c)|_{\kt \cap \km}-\rho^M}),
\end{multline}
where (for every $w \in W_K$),   an element $w_M \in W_M$ is chosen such that $w_M w (\lambda + \rho_c)|_{\kt \cap \km}$ is strictly dominant for $M$, and $\sigma_{w_M w (\lambda + \rho_c)|_{\kt \cap \km}-\rho^M} \in \hat M$ has highest weight $w_M w (\lambda - \rho_c)|_{\kt \cap \km}-\rho^M$. 

If $G/K$ is real hyperbolic space of dimension at least $4$, then 
\[
\tau_{\rem}^{\Delta_{\kp} \otimes W}\circ \Psi_*(\ind_{S, G}(D_W) ) = 0.
\]
\end{enumerate}
\end{theorem}

\begin{remark}
The condition that $\lambda + \rho_c$ is regular in part (d) of Theorem \ref{thm index non-ss} may be weakened to the more general but less explicit condition that $c_{\lambda}(\sigma)>0$ for all $\sigma \in \hat M$ such that $[\Delta_{\kp} \otimes W:\sigma^*]>0$.  
 See Proposition  6.5 in \cite{BM83}.
\end{remark}

Theorems \ref{thm index ss}, \ref{thm index ss higher} and \ref{thm index non-ss} are proved in Subsection \ref{sec pf ind thm}. The proof of Theorem \ref{thm index non-ss} is taken from \cite{BM83}. We give an independent, index-theoretic proof of Theorem \ref{thm index ss} in Section \ref{sec ss contr}. In the setting of Theorem \ref{thm index ss higher}, the computations from \cite{BM83} do not apply directly, and we use results from \cite{HST20} combined with arguments from the proof of Theorem \ref{thm index ss}.

\section{Traces on a Schwartz algebra}

We consider the setting of Subsection \ref{sec GK ex}, and initially do not assume that $G$ has real rank $1$.

\subsection{Schwartz spaces}\label{sec Schwartz spaces}

\begin{proof}[Proof of Lemma \ref{lem alg A infty}]
Let $a_0$ and $C$ be as in  \eqref{eq unif exp}, for $X=G$.
Let $f_1, f_2 \in C^{\infty}_c(G)$, $a\geq a_0$ and $P_1, P_2 \in \cU(\kg)$. 
Then, as in the proof of Lemma 2.5 in \cite{GHW25a}, 
\[
\|f_1 * f_2\|_{a, P_1, P_2} = \| L(P_1)f_1 *R(P_2)f_2\|_{a,1} 
\leq 
C  \|f_1 \|_{2a,P_1, 1} \|f_2\|_{a,1, P_2}. 
\]
If $a<a_0$, then
\[
\|f_1 * f_2\|_{a, P_1, P_2} \leq \|f_1 * f_2\|_{a_0, P_1, P_2} \leq  C  \|f_1 \|_{2a_0,P_1, 1} \|f_2\|_{a_0,1, P_2}.
\]
So $\calS(G)$ is closed under convolution, and convolution is a continuous operation on $\calS(G)$. 
\end{proof}

\begin{proof}[Proof of Lemma \ref{lem AEG}]
The following facts can be verified by writing out the various  definitions.
\begin{enumerate}
\item 
For all $\kappa \in \Gamma^{\infty}(E \boxtimes E^*)^G$, the relation \eqref{eq def Psi} indeed yields a well-defined function $\Psi(\kappa) \in (C^{\infty}(G) \otimes \End(V))^{K \times K}$.
\item The map $\Psi$ is linear.
\item  For $f \in (C^{\infty}(G) \otimes \End(V))^{K \times K}$, define $\kappa_{f} \in \Gamma^{\infty}(E \boxtimes E^*)$ by
\beq{eq def kappaf}
\kappa_{f}(gK, g'K)[g,v] = [g, f(g^{-1}g')v],
\eeq
for $g,g' \in G$ and $v \in V$. This yields a well defined section $\kappa_f \in \Gamma^{\infty}(E \boxtimes E^*)^G$.
\item The map $f\mapsto \kappa_f$ from the previous point is a two-sided inverse of $\Psi$.
\item The map $\Psi$ restricts to a bijection from $\Gamma^{\infty}_{\fp}(E \boxtimes E^*)^G$ to $(C^{\infty}_c(G) \otimes \End(V))^{K \times K}$.
\item 
For all $\kappa, \lambda \in  \Gamma^{\infty}_{\fp}(E \boxtimes E^*)^G$,  we have
 $\Psi(\kappa \lambda) = \Psi(\kappa) \Psi(\lambda)$.
 \item
 For all $f \in (C^{\infty}(G) \otimes \End(V))^{K \times K}$, all $P_1, P_2 \in \cU(\kg)$ and all $g,g' \in G$,
 \[
 (R(P_1) \otimes R(P_2) q^*\kappa_f) (g,g') = (L(P_1)R(P_2)f)(g^{-1}g').
 \]
\end{enumerate}
By left $G$-invariance of the Riemannian distance on $G$, the last point implies that for all $a>0$ and $P_1, P_2 \in \cU(\kg)$,
\[
\|\kappa_f\|_{a, R(P_1) \otimes R(P_2) } 
= \|f\|_{a,P_1, P_2}.
\]
So $\Psi$ indeed defines an isometric algebra isomorphism
\[
\Psi\colon \cA_{S, q}(E)^G \xrightarrow{\cong} (\calS(G) \otimes \End(V))^{K \times K}.
\]
\end{proof}

We will use a relation between $\calS(G)$ and Harish-Chandra's $L^p$-Schwartz spaces.
Let $\Xi$ be the matrix coefficient for a unit vector in the representation induced from the trivial representation of a minimal parabolic.
Fix $p>0$.
For $f \in C^{\infty}_c(G)$, $a>0$ and $P_1, P_2 \in \cU(\kg)$, we write 
\[
\|f\|_{\cC^p; a,P_1,P_2} := \sup_{g \in G} \bigl(1+d_{G/K}(eK, gK)\bigr)^{a} \Xi(x)^{-2/p}  |(L(P_1)R(P_2)f)(g)|.
\]
Let $\cC^p(G)$ be the completion of $C^{\infty}_c(G)$ in the seminorms $\| \cdot \|_{\cC^p; a,P,Q}$. 

\begin{remark}
The space $\cC^p(G)$ \label{page Cp} is the $L^p$-version of Harish-Chandra's Schwartz space. See e.g.\ 
 pages 161--162 of \cite{BM83}. For $p=1$, this is reduces to the definition on page 34 of \cite{Warner79}; for equivalence of the seminorms  used, see Proposition 7.15 and (7.49)--(7.51) in \cite{Knapp}. The space $\cC^1_{\varepsilon}(G)$ on page 98 of \cite{Warner79} equals $\cC^{2/(2+\varepsilon)}(G)$.
\end{remark}

\begin{lemma}\label{lem incl SC}
The space $\calS(G)$
is contained in $\cC^p(G)$ for all $p>0$, and the inclusion map is continuous.
\end{lemma}
\begin{proof}
%
The weight functions in the definition of the seminorms on $\cC^p(G)$ can be bounded by functions that increase exponentially in the distance to the identity element. 
\end{proof}

\begin{remark}\label{rem why SG}
At this point, let us comment on why we use the algebra $\calS(G)$ in the construction of the higher index on $G/K$. The spaces $\cC^1(G)$ and $\cC^2(G)$ are in fact Fr\'echet algebras as well, see pages 49--51 of \cite{Warner79} for the former, and Proposition 12.16 in \cite{Knapp} for the latter. These have two disadvantages for our purposes, however.
\begin{enumerate}
\item While the geometric traces from Subsection \ref{sec prelim traces} extend continuously to $(\cC^2(G) \otimes \End(V))^{K \times K}$, the trace \eqref{eq def T} 
is  only continuous on $(\cC^p(G) \otimes \End(V))^{K \times K}$ for $0<p<1$, hence not necessarily on the algebras  $(\cC^1(G) \otimes \End(V))^{K \times K}$ and  $(\cC^2(G) \otimes \End(V))^{K \times K}$. 
%
 \item It was shown in Lemma 3.8 in \cite{HST20} that the image of 
 the idempotent in \eqref{eq index idempotent} under the map $\Psi$ from Lemma \ref{lem AEG} lies in the unitisation of $(\cC^2(G) \otimes \End(V))^{K \times K}$, but the Fourier transform techniques used there do not seem to generalise easily to an argument that it lies in $(\cC^p(G) \otimes \End(V))^{K \times K}$ for smaller $p$. The heat operators on the diagonal in \eqref{eq index idempotent}  were shown to map into $(\cC^p(G) \otimes \End(V))^{K \times K}$ for all $p>0$ in Proposition 2.4 in \cite{BM83}. An extension of these arguments  to the off-diagonal elements in \eqref{eq index idempotent}  may be possible, but in view of the proof of Lemma 2.3 in \cite{BM83}, we expect this to be as involved as the proof that they map into $(\calS(G) \otimes \End(V))^{K \times K}$ in this paper and \cite{GHW25a}.
\end{enumerate}
So two key advantages of the algebra $\calS(G)$ over $\cC^1(G)$ and $\cC^2(G)$ are that 
\begin{enumerate}
\item  the functionals in Subsection \ref{sec spec tr} are continuous traces on  $(\calS(G) \otimes \End(V))^{K \times K}$, and
\item the idempotent in  \eqref{eq index idempotent} maps into $(\calS(G) \otimes \End(V))^{K \times K}$,
\end{enumerate}
as shown in the present paper.

Furthermore, the algebra $\cA_{S, \varphi}(E)^H$, and hence the higher index of Definition \ref{def H index},  can be defined for more general spaces than $G/K$, see \cite{GHW25a}.
\end{remark}

\subsection{Geometric traces on $\calS(G)$}\label{sec traces}

In this subsection, we prove Propositions \ref{prop tr SG}, \ref{prop cocycle SG} and \ref{prop Tlam} and Lemma \ref{lem T2lam}.
We start with a precise form of the statement that a conjugation-invariant functional is a trace.
\begin{lemma}\label{lem conj trace}
Let $H$ be a unimodular Lie group. Let $\cA(H)$ be a locally convex convolution algebra of functions on $H$, containing $C^{\infty}_c(H)$ as a dense subalgebra. Suppose that convolution is continuous as a map from $\cA(H) \times \cA(H)$ to $\cA(H)$.
Let $\tau\colon \cA(H) \to \C$ be a conjugation-invariant, continuous linear functional. Suppose that for all $f_1, f_2 \in C^{\infty}_c(H)$,
\beq{eq tau f1f2}
\tau(f_1 * f_2) = \int_H \tau\bigl(x\mapsto f_1(xy^{-1})f_2(y) \bigr)\, dy.
\eeq
Then $\tau$ is a trace on $\cA(H)$.
\end{lemma}
\begin{proof}
Let $f_1, f_2 \in C^{\infty}_c(H)$. Then by  \eqref{eq tau f1f2} and conjugation-invariance of $\tau$,
\[
\tau(f_1 * f_2) = \int_H \tau\bigl(x\mapsto f_1(y^{-1}x))f_2(y) \bigr)\, dy.
\]
By a substitution $y \mapsto  y^{-1}x$ (where we use unimodularity of $H$) and \eqref{eq tau f1f2}, this equals 
\[
\int_H \tau\bigl(x\mapsto f_2(xy^{-1})f_1(y) \bigr)\, dy =
\tau(f_2*f_1). 
\]
So $\tau$ is a trace on $C^{\infty}_c(H)$. By continuity of convolution and of $\tau$, this implies that  $\tau$ is a trace on $\cA(H)$.
\end{proof}
%
%
%
\begin{proposition}\label{prop tr SG F}
If $g$ is semisimple, then $\tau_g$ defines a continuous trace on $\calS(G)$.
\end{proposition}
\begin{proof}
If $g$ is semisimple, then 
$\tau_g$ is a continuous functional on $\cC^2(G)$; see Theorem 2.1  in \cite{HW2}, which is a consequence of Theorem 6 in \cite{HCDSII}. By Lemma \ref{lem incl SC}, $\tau_g$ restricts to a continuous functional on $\calS(G)$. (For a more direct argument, $\tau_g$ is a continuous functional on $\calS(G)$ by (3.4.4), (3.4.36) and (3.4.44) in \cite{BismutHypo}.) It is clearly conjugation-invariant, and it has the property \eqref{eq tau f1f2} by the Fubini-Tonelli theorem. Convolution is continuous on $\calS(G)$ by Lemma \ref{lem alg A infty}. So the claim follows by Lemma \ref{lem conj trace}.
\end{proof}
Proposition \ref{prop tr SG F} implies Proposition \ref{prop tr SG}. 

\begin{proof}[Proof of Proposition \ref{prop cocycle SG}]
By Theorem 3.5 in \cite{ST19}, $\Phi_{P, g}$ is a continuous cyclic cocycle on $\cC^2(G)$. So the claim follows from Lemma \ref{lem incl SC}. (Analogously to the proof of Proposition \ref{prop tr SG F}, we expect a more direct proof based on inequalities from \cite{BismutHypo} to be possible.)
\end{proof}

For the rest of this section, we assume that $G$ has real rank $1$. 
It is noted on page 299 of \cite{OW78} that $\tau_{n_0} + \tau_{n_0^{-1}}$ is a continuous linear functional on $\cC^2(G)$. In fact, the two terms separately define continuous traces on $\calS(G)$; we write out the proof here to show that this is particularly straightforward for the algebra $\calS(G)$.
\begin{lemma}\label{lem T2lam SG}
For all $n \in N_{2\lambda}$, the orbital integral $\tau_n$ defines a continuous trace on $\calS(G)$.
\end{lemma}
\begin{proof}
We first show that $\tau_n$ is a continuous linear functional on $\calS(G)$.
Note that $N<G_{n}$, because $n \in N_{2\lambda}$. So for all $f \in \calS(G)$, \eqref{eq Haar KAN} implies that
\[
|\tau_n(f)| \leq \int_{G/N}| f(xnx^{-1})|\, d(xN)
= \int_{KA} |f(kana^{-1}k^{-1})| a^{2\rho_N}\, dk\, da.
\]
For all $c>0$, this is smaller than or equal to
\beq{eq taun 1}
\|f\|_{c, 0} \int_{KA} e^{-cd(e, kana^{-1}k^{-1})} a^{2\rho_N}\, da \leq \|f\|_{c, 0}  e^{c \diam(K)}\int_{A} e^{-cd(e, ana^{-1})} a^{2\rho_N}\, da
\eeq
by the triangle inequality. If $n = \exp(Y)$ for $Y  \in \kn_{2\lambda}$, then the distance from $e$ to $ana^{-1}$ in the abelian group $N_{2\lambda}$ is
\[
d_{N_{2\lambda}}(1, \exp(\Ad(a)Y)) = \|\Ad(a)Y\| = a^{2\lambda} \|Y\|.
\]
This is the same as the distance from $e$ to $ana^{-1}$ in $G$, because $t \mapsto \exp(tY)$ is a geodesic. So the right hand side of \eqref{eq taun 1} is smaller than or equal to
\[
 \|f\|_{c, 0}  e^{c \diam(K)} \|Y\| \int_{A} a^{-2c\lambda} a^{2\rho_N}\, da.
\]
By \eqref{eq def rho}, the integral converges if $c>\dim(\kn_{\lambda})/2+ \dim(\kn_{2\lambda})$. So $\tau_n$ is indeed continuous.

It is immediate that $\tau_n$ is conjugation-invariant. And it satisfies \eqref{eq tau f1f2} for compactly supported functions, by the Fubini-Tonelli theorem. So Lemma \ref{lem conj trace} implies the claim.
\end{proof}
Lemma \ref{lem T2lam SG} implies Lemma \ref{lem T2lam}.

\begin{proposition}\label{prop Tlam SGF}
The integral \eqref{eq def Tlambda} converges for all $f \in \calS(G)$, and defines a continuous trace on this algebra.
\end{proposition}
\begin{proof}
The integral  defines a continuous functional on $\cC^2(G)$, as noted on page 299 of \cite{OW78}. So  by Lemma \ref{lem incl SC}, it is a continuous functional on $\calS(G)$.

For the trace property, we use Lemma \ref{lem conj trace}. Let us show that $\tau_{\lambda}$ is conjugation-invariant.\footnote{This seems to be known; the functional $\tau_{\lambda}$ is called an `invariant term' in \cite{BM83} (and denoted by $T_1$ there). We include a proof here because we could not immediately find a reference.} Because the action by $AN$ on $G$ by right multiplication is proper and cocompact, there exists a function
$\varphi \in C_c(G)$ such that for all $x \in G$,
\beq{eq phi AN}
\int_{AN} \varphi(xan) \, da\, dn = 1.
\eeq
The group $AN$ normalises $N$, and for all $f \in C_c(G)$, and all $a \in A$ and $n' \in N$,
\[
a^{2\rho_N}  \int_{N} f(an'n (an')^{-1})\, dn = \int_N  f(n)\, dn,
\]
where $a^{2\rho_N}$ is a Jacobian.
 It follows
that
for any  function $\varphi$ as in \eqref{eq phi AN}, and all $f \in C_c^{\infty}(G)$,
\beq{eq tau lam 1}
\begin{split}
\tau_{\lambda}(f) &= \int_N \int_{KAN} \varphi(kan')  f(kn k^{-1}) \, dk\, da\, dn'\, dn\\
&= \int_N \int_{KAN} \varphi(kan') f(kan' n (kan')^{-1}) a^{2\rho_N}\, dk\, da\, dn'\, dn\\
&= \int_N \int_{G} \varphi(g) f(g n g^{-1}) \, dg \, dn,
\end{split}
\eeq
by \eqref{eq Haar KAN}. 
For $h \in G$, let $f_h \in C_c^{\infty}(G)$  be given by $f_h(g) = f(hgh^{-1})$. Then by \eqref{eq tau lam 1} and a substitution $hg \mapsto g$,
\beq{eq tau lam 2}
\tau_{\lambda}(f_h) = \int_N \int_{G} (L_{h}\varphi)(g) f(g n g^{-1}) \, dg \, dn.
\eeq
The function $L_h\varphi$ has the same property \eqref{eq phi AN} as $\varphi$, so \eqref{eq tau lam 2} equals $\tau_{\lambda}(f)$. 

We have seen that $\tau_{\lambda}$ is continuous and conjugation-invariant. It satisfies \eqref{eq tau f1f2} by the Fubini--Tonelli theorem, so it is a trace on $\calS(G)$ by Lemma \ref{lem conj trace}.
\end{proof}
Proposition \ref{prop Tlam SGF} implies Proposition \ref{prop Tlam}.

%

\subsection{Spectral traces on $\calS(G)$}\label{sec trace Gamma}

In this subsection, we prove Lemma \ref{lem tau res}, Propositions \ref{prop Rd cts} and \ref{prop trace T}, and Corollary  \ref{cor Rd trace}.

\begin{proof}[Proof of Lemma \ref{lem tau res}]
The functional \eqref{eq tau res} is continuous on $(\cC^p(G) \otimes \End(V))^{K \times K}$ for $0<p<1$, as pointed out below Theorem 8.4 in \cite{Warner79}. Hence it is continuous on $(\calS(G) \otimes \End(V))^{K \times K}$ by Lemma \ref{lem incl SC}. It is immediate from the definition that $\tau_{\rem}$ is conjugation-invariant, so the claim follows by Lemma \ref{lem conj trace}.
\end{proof}

\begin{proposition}\label{prop Rd cts F}
Let $F \subset \hat K$ be a finite set, and let $\cC^1_{F}(G) \subset \cC^1(G)$ be the subspace of isotypical components corresponding to $F$ for the right $K$-action. 
For all $f \in \cC^1_{F}(G)$, the operator $R_d^{\Gamma}(f)$ is trace-class. 
The linear map
\beq{eq RdGamma L2d C1}
R_d^{\Gamma}\colon \cC^1_{F}(G) \to \calL^1(L^2_d(\Gamma \backslash G)) 
\eeq
is continuous.
\end{proposition}
\begin{proof}
The first claim is Theorem 4.6 in \cite{Warner79}.

Let $L^2_0(\Gamma \backslash G) \subset L^2_d(\Gamma \backslash G)$ be the $G$-invariant subspace of cusp forms. By Theorem 4.3 in \cite{Warner79}, the representation $R^{\Gamma}$ defines a linear map
\beq{eq R0Gam}
R_0^{\Gamma}\colon \cC^1(G) \to \calL^1(L^2_0(\Gamma \backslash G)). 
\eeq
On pages 39--40 of \cite{Warner79}, it is shown that $f \mapsto \Tr(R^{\Gamma}_0(f))$ is a continuous functional on  $\cC^1(G)$. In fact, the arguments given there can be strengthened slightly, to show that  \eqref{eq R0Gam} is continuous. Indeed, 
 apart from replacing the left  quasiregular representation used in \cite{Warner79} by the  right  quasiregular representation we use here, 
the only changes that need to be made are that in the second-last set of displayed equations on page 40 of \cite{Warner79}, one should take absolute values of $L^0_{G/\Gamma}(\alpha)$, $(U(\alpha)e_{U, i}, e_{U, i})$ and the integral over $G$ in the third line. Then the bottom equation on page 40 of \cite{Warner79} becomes $\tr(| L_{G/\Gamma}^0(\alpha)|) \leq C \| \Delta_L^{2N}  \alpha\|_1$, and continuity of \eqref{eq R0Gam} follows.
 
Let $L^2_0(\Gamma \backslash G)^{\perp} \subset L^2_d(\Gamma \backslash G)$ be the orthogonal complement to $L^2_0(\Gamma \backslash G)$. It is shown on page 32 of \cite{Warner79} that the representation of $G$ on  this space is $K$-finite. Hence the image of the compression
\beq{eq 0Rd}
(R_d^{\Gamma})^{\perp}\colon \cC^1_{F}(G) \to \calL^1(L^2_0(\Gamma \backslash G)^{\perp}) 
\eeq
of \eqref{eq RdGamma L2d C1} 
is finite-dimensional. Let $n_{F}$ be its dimension. Because $\cC^1(G)$ embeds continuously into $L^1(G)$, we find that for all $f \in \cC^1_{F}(G) $,
\[
\Tr(|(R_d^{\Gamma})^{\perp}(f)|) \leq n_{F} \|(R_d^{\Gamma})^{\perp}(f)\|_{\cB(L^2_d(\Gamma \backslash G)^{\perp})} \leq n_{F} \|f\|_{L^1(G)},
\]
and  \eqref{eq 0Rd} is also continuous.

Continuity of \eqref{eq R0Gam} and \eqref{eq 0Rd} implies continuity of \eqref{eq RdGamma L2d C1}.
\end{proof}
Proposition \ref{prop Rd cts F} and Lemma \ref{lem incl SC} imply Proposition \ref{prop Rd cts}.

\begin{remark}
In the proof of Proposition \ref{prop Rd cts F}, we used Theorem 4.6 in \cite{Warner79}, which states that $R_d^{\Gamma}(f)$ is trace-class for all $K$-finite $f \in \cC^1(G)$. This was in fact generalised to $G$ of arbitrary real rank by M\"uller \cite{Muller89}. This was further generalised to  $f \in \cC^1(G)$ that are not necessarily  $K$-finite by M\"uller \cite{Muller98} and Ji \cite{Ji98}.
\end{remark}


Now consider a finite subset $F \subset \widehat{K \times K}$. Let $\calS_F(G) \subset \calS(G)$ be the subalgebra of functions  in isotypical components corresponding to $F$, with respect to the left and right regular representations.
Lemma \ref{lem incl SC} and Proposition \ref{prop Rd cts F} have the following consequence. This implies Corollary \ref{cor Rd trace}.
\begin{corollary}\label{cor Rd trace F}
The map $f \mapsto \Tr(R_d^{\Gamma}(f))$ is a continuous trace on $\calS_F(G)$.
\end{corollary}

\begin{remark}
To prove Corollary \ref{cor Rd trace F}, we could have used the result from \cite{Warner79} that $f \mapsto \Tr(R^{\Gamma}_0(f))$ is a continuous functional on  $\cC^1(G)$, rather than the stronger Proposition \ref{prop Rd cts F}. We have included this proposition, because a consequence, Proposition \ref{prop Rd cts},  is also used in Corollary \ref{cor DGamma}.
\end{remark}

%
The following result is a consequence of Theorem  8.4 in \cite{Warner79} and  the theorem on page 299 of \cite{OW78} and the comments below it.
\begin{theorem}[Osborne--Warner]\label{thm STF}
For all $0<p<1$, 
the expression \eqref{eq def T} 
%
defines a continuous linear functional $\tau_{\rem}$ on $\cC^p(G)$.
%
%
\end{theorem}


\begin{corollary}\label{cor trace T F}
The functional $\tau_{\rem}$ from Theorem \ref{thm STF} restricts to a continuous trace on $\calS(G)$.
\end{corollary}
\begin{proof}
The restriction of $\tau_{\rem}$ to $\calS(G)$ is continuous by Lemma \ref{lem alg A infty} and Theorem \ref{thm STF}. It is a trace by Propositions \ref{prop tr SG} and \ref{prop Tlam}  and Lemma \ref{lem T2lam} .
\end{proof}
Corollary \ref{cor trace T F} implies Proposition \ref{prop trace T}.

\section{Contributions from semisimple elements}\label{sec ss contr}

To prove Theorem \ref{thm index ss}, one can use the computations in Section 4.1 of \cite{BM83}. Those computations are based on some nontrivial results from representation theory. We give an independent, index-theoretic proof here, based on a generalisation of the main result in \cite{HW4}. We hope that this approach and some of the intermediate results may be of independent interest. The index-theoretic arguments also extend to a proof of Theorem \ref{thm index ss higher}.

For a locally compact group $H$ acting properly and cocompactly on a manifold, we will use 
the standard equivariant index 
\beq{eq def H index}
\ind_H(D) \quad \in K_0(C^*_r(H))
\eeq
of  $H$-equivariant elliptic operators $D$. 
There are many equivalent definitions, see e.g.\ \cite{Connes94, Kasparov83, WY20}.

\subsection{A holomorphic  fixed-point formula}

The first step is a version of the fixed-point formula in \cite{HW2} for Dolbeault--Dirac operators.
Let $X$ be a complex manifold. Let $G$ be a connected, real semisimple Lie group acting properly  and cocompactly on $X$ by holomorphic maps. Let $E \to M$ be a holomorphic, Hermitian, $G$-equivariant vector bundle. Let $\bar \partial_E$ be the Dolbeault operator on antiholomorphic differential forms twisted by $E$, and let $\bar \partial_E^*$ be its formal adjoint. Then we have the Dolbeault--Dirac operator $\bar \partial_E + \bar \partial_E^*$.

Let $g \in G$ be semisimple. 
Because $G$ acts holomorphically on $X$, the fixed point set $X^g \subset X$ of $g$ is a complex submanifold, and the normal bundle $N \to X^g$ is a complex vector bundle.  (To be precise, the connected components of $X^g$ are  complex submanifolds of possibly different dimensions, so the rank of $N$ may jump between these components.) 
 Its complexification decomposes into holomorphic and antiholomorphic parts as $N \otimes \C = N^{1,0} \oplus N^{0,1}$. Fix a Haar measure $dz$ on the centraliser $G_g$. Let $\chi^g \in C^{\infty}_c(X^g)$ be such that for all $x \in X^g$,
\[
\int_{G_g} \chi^g(zx)\, dz = 1.
\]
(Here we use compactness of $X^g/Z$; see Lemma 2.4 in \cite{HW2}.)

The restriction $E^g \to X^g$  of $E$ is a holomorphic vector bundle over the complex manifold $X^g$, so we have the Dolbeault--Dirac operator $\bar \partial_{E^g} + \bar \partial_{E^g}^*$ on $X^g$. For every   connected component $X_j^g$ of $X^g$, we write $E_j^g := E|_{X_j^g}$. We similarly have the Dolbeault--Dirac operator  $\bar \partial_{E^g_j} + \bar \partial_{E^g_j}^*$ on $X^g_j$.

Consider the equivariant indices
\[
\begin{split}
\ind_G(\bar \partial_E + \bar \partial_E^*) &\in K_0(C^*_r(G));\\
\ind_{G_g}(\bar \partial_{E^g_j} + \bar \partial_{E^g_j}^*) &\in K_0(C^*_r(G_g)).
\end{split}
\]
The orbital integral trace $\tau_g$ is a well defined map on $K_0(C^*_r(G))$;  see Section 2.1 of \cite{HW2}. The von Neumann trace $\tau_e$ defines a map on $K_0(C^*_r(G_g))$. We write  $\overline{g^{\Z}}$ for the closure of the cyclic subgroup of $G$ generated by $g$. 
\begin{theorem}\label{thm fixed Dolb}
\begin{enumerate}
\item[(a)]
If $g$ does not lie in a compact subgroup of $G$, then $\tau_g(\ind_G(\bar \partial_E + \bar \partial_E^*) ) = 0$.
\item[(b)] 
If 
 $g$ lies in a compact subgroup of $G$,
then
\[
\tau_g(\ind_G(\bar \partial_E + \bar \partial_E^*) ) =
 \int_{T(X^g)} \chi^g \frac{\ch\left(\left[\sigma_{\bar \partial_{E^g} + \bar \partial_{E^g}^*}\right](g)\right) \Todd(T(X^g) \otimes \C)}{\ch\left( \left[\Bigwedge N^{0,1}\right](g) \right)}.
 \]
\item[(c)] Suppose that $g$ lies in a compact subgroup of $G$. 
Suppose, furthermore,  that on each connected component $X_j^g$ of $X^g$, the group element $g$ acts on every fibre of $E_j^g$ as the same scalar $g|_{E_j^g}$.
Suppose also that the  bundle $N^{0,1}$ is trivial as a $\overline{g^{\Z}}$-equivariant vector bundle. Then
\[
\tau_g(\ind_G(\bar \partial_E + \bar \partial_E^*) ) = \sum_j 
 \frac{g|_{E_j^g}}{\det_{\C}(1-g|_{N^{0,1}}) |_{X_j^g} }\tau_e(\ind_{G_g}(\bar \partial_{E_j^g} + \bar \partial_{E_j^g}^*)).
\]
\end{enumerate}
\end{theorem}
\begin{proof}
Part (a) follows from the second statement in  Theorem 2.5 in \cite{HW2}. (This in fact holds for all semisimple $g \in G$ not contained in compact subgroups. The ``finite Gaussian orbital integral'' condition used there holds for general semisimple group elements, by the equalities from \cite{BismutHypo} also mentioned in the proof of Proposition \ref{prop tr SG F}.)


To prove part (b), we again apply Theorem 2.5 in \cite{HW2}. This now yields
\[
\tau_g(\ind_G(\bar \partial_E + \bar \partial_E^*) ) =
 \int_{T(X^g)} \chi^g 
 \frac{\ch\left(\left[\sigma_{\bar \partial_{E} + \bar \partial_{E}^*}\right] (g)\right) \Todd(T(X^g) \otimes \C)}{\ch\left(\left[\Bigwedge N\otimes \C\right](g)\right)}.
\]
To be precise,  in Theorem 2.5 in \cite{HW2} it is stated that 
this equality holds for almost all elliptic elements $g$ not contained in compact subgroups. The result applies to ``almost all'' such $g$ because of the ``finite Gaussian orbital integrals'' condition. But this  turns out to hold in general, see equations (3.3.36), (3.4.4) and (4.2.9) in \cite{BismutHypo}.

We use the decomposition
\[
\left[\sigma_{\bar \partial_{E} + \bar \partial_{E}^*}|_{X^g}\right] (g) =  \left[\Bigwedge (N^{0,1})^*\right](g) \otimes \left[\sigma_{\bar \partial_{E^g} + \bar \partial_{E^g}^*}\right](g) = 
 \left[\Bigwedge N^{1,0}\right](g) \otimes \left[\sigma_{\bar \partial_{E^g} + \bar \partial_{E^g}^*}\right](g)
\]
By multiplicativity of the Chern character, 
\[
\frac{\ch\left( \left[\Bigwedge N^{1,0}\right](g) \right)}{\ch\left(\left[\Bigwedge N\otimes \C\right](g)\right)} = \frac{1}{\ch\left( \left[\Bigwedge N^{0,1}\right](g) \right)},
\]
so (b) follows.

To deduce (c) from (b), we first note that  the assumption on  the way $g$ acts on $E|_{X^g}$ implies that
\[
\left[\sigma_{\bar \partial_{E^g} + \bar \partial_{E^g}^*}|_{X_j^g}\right](g) = g|_{E_j^g} \left[\sigma_{\bar \partial_{E_j^g} + \bar \partial_{E_j^g}^*}\right],
\]
for every connected component $X_j^g$ of $X^g$. 
And if $N^{0,1}$ is trivial, then for all $j$, we have $N^{0,1}|_{X_j^g} = X^g_j \times W_j$ for some representation $W_j$ of the compact abelian group $\overline{g^{\Z}}$. (The dimension of $X^g_j$, and hence of $W_j$, may depend on $j$.) Then 
\[
[\Bigwedge N^{0,1}|_{X_j^g}] = [X_j^g \times \C] \otimes \left[\Bigwedge W_j \right] \quad \in K^0_{\overline{g^{\Z}}}(X_j^g) \cong K^0(X_j^g) \otimes R(\overline{g^{\Z}}).
\]
Here $\Bigwedge W_j$ should be interpreted as a graded representation, i.e.\ 
\[
\left[\Bigwedge W_j \right] = \Bigl[\Bigwedge^{\even} W_j \Bigr]  - \Bigl[\Bigwedge^{\odd} W_j \Bigr]  \in R(\overline{g^{\Z}}).
\]
So
\[
\ch\left( \left[\Bigwedge N^{0,1}\right](g) \right)= \tr \left( g|_{\bigwedge^{\even} W_j} \right) -  \tr \left( g|_{\bigwedge^{\odd} W_j} \right) = {\det}_{\C}(1-g|_{W_j}) = {\det}_{\C}(1-g|_{N^{0,1}})|_{X_j^g}.
\]
We find that
\begin{multline*}
\tau_g(\ind_G(\bar \partial_E + \bar \partial_E^*) ) =\\
\sum_j \left. \frac{\tr(g|_E)}{{\det}_{\C}(1-g|_{N^{0,1}})}\right|_{X_j^g}
 \int_{T(X_j^g)} \chi^g|_{X_j^g} {\ch\left(\left[\sigma_{\bar \partial_{E_j^g} + \bar \partial_{E_j^g}^*}\right]\right) \Todd(T(X_j^g) \otimes \C)}.
\end{multline*}
By Proposition 4.4 and Theorem 6.12 in \cite{Wang14}, the integral on the right equals $\tau_e(\ind_{G_g}(\bar \partial_{E_j^g} + \bar \partial_{E_j^g}^*))$.
\end{proof}

\begin{example}
If $X$ and $G$ are compact, and $X^g$ is finite, then Theorem \ref{thm fixed Dolb}(c) reduces to
\[
\ind_G(\bar \partial_E + \bar \partial_E^*) (g) = \sum_{x \in X^g} 
 \frac{g|_{E_x}}{\det_{\C}(1-g|_{T_xX^{0,1}})  } =  \sum_{x \in X^g} 
 \frac{g|_{E_x}}{\det_{\C}(1-g^{-1}|_{T_xX^{1,0}})  }.
\]
This is consistent with Atiyah and Bott's holomorphic fixed-point theorem, Theorem 4.12 in \cite{ABII}, because in the notation of that paper,
\[
\ind_G(\bar \partial_E + \bar \partial_E^*) (g) = L(g^{-1}, g).
\]
\end{example}

\subsection{A  fixed-point set}

Until the end of Subsection \ref{sec fixed pt GT}, we suppose that $T<K$ is a compact Cartan subgroup of $G$, so $\rank(G) = \rank(K)$.  We fix a positive root system $R^+(G,T)$ for $(G, T)$. This determines a $G$-invariant complex structure on $G/T$ such that $T_{eT}^{1,0}(G/T) = \bigoplus_{\alpha \in R^+(G,T)} \kg^{\C}_{\alpha}$. Let $g \in T$. We will apply Theorem \ref{thm fixed Dolb} to the space $X = G/T$, and we start by analysing the fixed-point set $(G/T)^g$.

\begin{lemma}\label{lem conj norm}
For all $x \in G$ such that $xgx^{-1} \in T$, there exists  $y \in N_G(T)$ such that $xgx^{-1} = ygy^{-1}$.
\end{lemma}
\begin{proof}
Because $xTx^{-1} \subset G_{xgx^{-1}}$, we have two maximal tori $T$ and $xTx^{-1}$ in the connected component $(G_{xgx^{-1}})_0$. These are conjugate in $(G_{xgx^{-1}})_0$. (See e.g.\ Propositions 6.59 and 6.61 in \cite{Knapp96}, or use the facts that any two maximal compact subgroups of $(G_{xgx^{-1}})_0$ are conjugate, and so are any two maximal tori inside such a compact subgroup.) Let $h \in (G_{xgx^{-1}})_0$ be such that $hxTx^{-1}h^{-1} = T$. Then $y:= hx$ has the desired properties.
\end{proof}

Recall that we used the Weyl groups $W_K = N_K(T)/T$ and $W_{K_{g}} = N_{K_g}(T)/T$. Note that  $G_g w$ is a right $T$-invariant subset of $G$ for all $w \in N_G(T)$.
\begin{lemma}\label{lem fixed GT}
The fixed-point set $(G/T)^g$ decomposes as
\beq{eq fixed GT}
(G/T)^g = \coprod_{w \in W_{K_g} \backslash W_K} G_gw/T,
\eeq
where the disjoint union is over representatives $w \in N_K(T)$ of classes in $W_{K_g} \backslash W_K$.
\end{lemma}
\begin{proof}
We first show that
\beq{eq fixed GT 1}
(G/T)^g = G_g N_G(T)/T.
\eeq
It is immediate that the right hand side is contained in the left hand side. Now let $xT \in (G/T)^g$. Then $x^{-1}gx  \in T$. By Lemma \ref{lem conj norm}, there is a $y \in N_G(T)$ such that $x^{-1}gx = y^{-1}gy$. So $xy^{-1} \in G_g$, and $xT = (xy^{-1})yT$ lies in the right hand side of \eqref{eq fixed GT 1}.

Next, we show that
\beq{eq fixed GT 2}
(G/T)^g = \coprod_{w \in N_{G_g}(T)\backslash N_G(T)} G_g w/T,
\eeq
where the disjoint union is over a set of representatives $w \in N_G(T)$. Indeed, if $y \in N_{G_g}(T) \subset G_g$ and $w \in N_G(T)$, then
$G_g yw/T = G_g w/T$. And if $w,w' \in N_G(T)$ and $G_g w/T = G_g w'/T$, then let $y := w'w^{-1} \in N_G(T)$. Then $G_g w/T = G_g yw/T$, so there are $a \in G_g$ and $b \in T$ such that $awb = yw$. Hence $a^{-1}y = wbw^{-1} \in T \subset G_g$. So $y \in G_g$, and therefore $y \in N_{G_g}(T)$. We find that $G_g w/T = G_g w'/T$ if and only if $w'w^{-1} \in N_{G_g}(T)$, which implies that the right hand side of \eqref{eq fixed GT 1} equals the right hand side of  \eqref{eq fixed GT 2}. 

Finally, we have $N_G(T) = N_K(T)$ and $N_{G_g}(T) = N_{K_g}(T)$; see for example Lemma 6.15 in \cite{HW}. 
So the decomposition \eqref{eq fixed GT} follows from \eqref{eq fixed GT 2} and the equality
\[
N_{K_g}(T)\backslash N_K(T) = (N_{K_g}(T)/T)\backslash (N_K(T)/T).
\]
\end{proof}

\begin{example}
If $g=e$, then Lemma \ref{lem fixed GT} reduces to the equality $G/T = G/T$.
At the other extreme, if the powers of $g$ are dense in $T$, then Lemma \ref{lem fixed GT} becomes $(G/T)^T = W_K$; this is Lemma 6.15 in \cite{HW}. 
\end{example}

\subsection{A normal bundle}

For $w \in N_G(T)$, let $J_w$ be the $G$-invariant complex structure on $G/T$ such that 
\[
T_{eT}^{1,0}(G/T) \cong \bigoplus_{\alpha \in R^+(G, T)} \kg^{\C}_{w\alpha}.
\]
Then $J_e$ is the complex structure we have considered so far.

\begin{lemma}\label{lem Rw holom}
Let $w \in N_G(T)$.   
The map $R_w\colon G/T \to G/T$, given by $R_w(xT) = xwT$ for $x \in G$, is a well-defined, $G$-equivariant diffeomorphism. It is holomorphic as a map
\beq{eq Rw holom}
R_w\colon (G/T, J_w) \to (G/T, J_e).
\eeq
\end{lemma}
\begin{proof}
The first claim follows directly from the definitions. We show that \eqref{eq Rw holom} is  holomorphic.

Note that $R_w = L_w \circ C_{w^{-1}}$, where $L$ is the standard action by $G$ on $G/T$, and $C_{w^{-1}} = L_{w^{-1}} \circ R_w$ is induced by conjugation by $w^{-1}$. The map $L_w$ is holomorphic for any of the complex structures $J_{w'}$, because they are $G$-invariant. And
\[
T_{eT}C_{w^{-1}} = \Ad(w^{-1})\colon \kg/\kt \to \kg/\kt.
\]
Because $\Ad(w^{-1})\kg^{\C}_{\alpha} = \kg^{\C}_{w^{-1}\alpha}$, this map induces a complex linear map from $(T_{eT}(G/J), J_w)$ to $(T_{eT}(G/J), J_e)$. For general $x \in G$, we have $C_{w^{-1}}  = L_{w^{-1}x w}  \circ C_{w^{-1}} \circ L_{x^{-1}}$. So
\[
T_{xT}C_{w^{-1}} = T_{eT} L_{w^{-1}xw} \circ T_{eT}C_{w^{-1}} \circ T_{xT} L_{x^{-1}}.
\]
The three maps on the right hand side are complex-linear, hence so is the map on the left. We find that the map \eqref{eq Rw holom} is indeed holomorphic.
\end{proof}

The complex structure $J_w$, for $w \in N_G(T)$, restricts to the normal bundle $N \subset T(G/T)|_{(G/T)^g}$ to $(G/T)^g$ in $G/T$. We denote this restriction by $J_w$ as well.
\begin{lemma}\label{lem Nw}
For every $w \in N_G(T)$, the map  $R_w \colon G_g/T \to G_g w/T$ given by $R_w(xT) = xwT$ for $x \in G_g$, is a well-defined, $G_g$-equivariant diffeomorphism. The vector bundles
\beq{eq normal Nw}
\begin{split}
(N|_{G_g/T}, J_w)& \to G_g/T;\\
R_w^*(N|_{G_gw/T}, J_{e}) &\to G_g/T
\end{split}
\eeq
are isomorphic as $G_g$-equivariant, holomorphic vector bundles. 
\end{lemma}
\begin{proof}
Let $w \in N_G(T)$.  The first claim follows from the first claim in Lemma \ref{lem Rw holom}.


%

We write $N(S)$ for the normal bundle of a submanifold $S \subset G/T$. Because the sets $G_gw/T$ are connected components of $(G/T)^g$, we have
\[
TR_{w}(N|_{G_g/T}) = TR_w(N(G_g/T)) = N(R_w (G_g/T)) = N(G_gw/T) = N|_{G_gw/T}.
\]
This fact, together with $G$-equivariance of $R_w$ and the fact that the map \eqref{eq Rw holom} is holomorphic by Lemma \ref{lem Rw holom}, implies that the bundles \eqref{eq normal Nw} are equivariantly  isomorphic as holomorphic vector bundles.
\end{proof}


%

\begin{lemma}\label{lem norm triv}
The  bundle $N^{0,1} \to (G/T)^g$ is trivial as a  $G_g$-vector bundle.
\end{lemma}
\begin{proof}
%
Let $w \in N_K(T)$.
We first prove that  $(N^{0,1}|_{G_g/T}, J_w) \to G_g/T$ is  $G_g$-equivariantly trivial. Indeed,
\[
N|_{G_g/T} = G_g\times_T \kg/\kg_g,
\]
where $\kg_g$ is the Lie algebra of $G_g$. 
So
\beq{eq N10}
(N^{0,1}|_{G_g/T}, J_w) = G_g\times_T ((\kg/\kg_g)^{0,1}, J_w),
\eeq
where
\beq{eq g gg 10}
((\kg/\kg_g)^{0,1}, J_w) = \bigoplus_{\alpha \in R^+(G, T) \setminus R^+(G_g, T)} \kg^{\C}_{-w\alpha}.
\eeq

The bundle \eqref{eq N10} is isomorphic as a $G_g$-equivariant, complex vector bundle to $G_g/T \times ((\kg/\kg_g)^{0,1}, J_w)$ via the map given by $[x,Y] \mapsto (xY, \Ad(x)Y)$, for $x\in G_g$ and $Y \in ((\kg/\kg_g)^{0,1}, J_w)$. So $(N^{0,1}|_{G_g/T}, J_w) \to G_g/T$ is equivariantly trivial as claimed. 

The claim now follows from Lemmas \ref{lem fixed GT} and \ref{lem Nw}.
\end{proof}

As in Theorem \ref{thm index ss}, let $\rho$ be half the sum of the positive roots in $R^+(G, T)$, and let $\rho_g$ be half the sum of the positive roots in $R^+(G_g, T)$. 
\begin{lemma}\label{lem det N}
For all $w \in W_K$,
\beq{eq det N}
{\det}_{\C}(1-g|_{N^{0,1}})|_{G_gw/T} =e^{w(\rho_g-\rho)}(g)  \det(w) e^{\rho-\rho_g}(g) \prod_{\alpha \in R^+(G, T) \setminus R^+(G_g, T)} (1-e^{-\alpha}(g)).
\eeq
\end{lemma}
\begin{proof}
Let $w \in N_K(T)$. Then by \eqref{eq N10} and \eqref{eq g gg 10}, 
\[
\begin{split}
 {\det}_{\C}(1-g|_{(N^{0,1},J_w  )})|_{G_g/T}  
&= \prod_{\alpha \in R^+(G, T) \setminus R^+(G_g, T)} (1-e^{-w\alpha}(g))\\
&= e^{w(\rho_g - \rho)}(g)   e^{w(\rho-\rho_g)}(g) \prod_{\alpha \in R^+(G, T) \setminus R^+(G_g, T)} (1-e^{-w\alpha}(g)).
\end{split}
\]
The right hand side with the first exponential factor removed is a ratio of Weyl denominators. This is alternating with respect to $w$. So 
the last expression equals
the right hand side of \eqref{eq det N}. The claim now follows by Lemma \ref{lem Nw}.
\end{proof}

\subsection{A fixed-point formula on $G/T$}\label{sec fixed pt GT}

For analytically  integral $\nu \in i\kt^*$, we write $\C_{\nu}$ for the complex numbers equipped with the unitary representation of $T$ with weight $\nu$.

\begin{lemma}\label{lem Rw Lmu}
Let $\mu \in i\kt^*$ be integral, and let $w \in N_K(T)$ be such that $w\cdot \mu = \mu \circ \Ad(w)^{-1}$ is integral. Then there is a $G_g$-equivariant isomorphism of complex line bundles
\[
R_w^*(G_g w/T \times_T \C_{\mu}) \cong G_g \times_T \C_{w \cdot \mu} \to G_g/T.
\]
\end{lemma}
\begin{proof}
We have
\[
R_w^*(G_g w/T \times_T \C_{\mu}) = \{(xT, [xwt, z]); x\in G_g, t \in T, z \in \C_{\mu}\}.
\]
The map from the right hand side to $G_g \times_T \C_{w\cdot \mu}$ mapping
 $(xT, [xwt, z])$ in this set to $[x, e^{\mu}(t)z]$, is a well-defined, $G_g$-equivariant isomorphism of complex line bundles.
\end{proof}

We consider a Dirac operator $D_W$ of the form \eqref{eq DW}.
Let $\lambda \in i\kt^*$ be the highest weight of $W$ with respect to the set of compact positive roots $R^+(K, T) \subset R^+(G, T)$. 
%
As in Theorem \ref{thm index ss}, 
let $\rho_c$ be half the sum of the positive compact roots in $R^+(K, T)$.  Let $\rho_n = \rho- \rho_c$.
For analytically integral $\nu \in i\kt^*$, consider the holomorphic line bundle $L_{\nu} := G \times_T \C_{\nu} \to G/T$. Let $d^{G_g}_{\mu}$ be as in \eqref{eq def dG}.
\begin{lemma}\label{lem ds Gg}
The element $\lambda - \rho_n \in i\kt^*$ is analytically integral. For all $w \in N_K(T)$, 
\beq{eq tau e Gg}
\tau_e\bigl(\ind_{G_g}(\bar \partial_{L_{\lambda - \rho_n}|_{G_g w/T}} + \bar \partial_{L_{\lambda - \rho_n}|_{G_g w/T}}^*)  \bigr)= (-1)^{\dim(G_g/K_g)/2}d^{G_g}_{w(\lambda + \rho_g - \rho_n)},
\eeq
where the Dolbeault operator on the left is defined with respect to the restriction of the complex structure $J_e$ on $G/T$ to the connected component  $G_g w/T$ of the complex submanifold $(G/T)^g\subset G/T$.
%
\end{lemma}
\begin{proof}
The element $\lambda - \rho_n \in i\kt^*$ is analytically integral, because $\Delta_{\kp} \otimes W$ defines a representation of $K$ with highest weight $\lambda + \rho_n$.


For $w \in N_K(T)$, let $J^{G_g}_w$ be the $G_g$-invariant complex structure on $G_g/T$ such that
\[
T^{1,0}_{eT}(G_g/T) = \bigoplus_{\alpha \in R^+(G_g, T)} \kg^{\C}_{w\alpha}.
\]
Let $w \in N_K(T)$. 
Then Lemma \ref{lem Rw holom} implies that the map $R_w$ restricts to a $G_g$-equivariant, holomorphic diffeomorphism
\[
R_w\colon (G_g/T, J^{G_g}_w) =  (G_g/T, J_w|_{G_g/T}) \to (G_gw/T, J_e|_{G_g w/T}). 
\]
And Lemma \ref{lem Rw Lmu} implies that
\[
R_w^*(L_{\lambda - \rho_n}|_{G_gw/T}) \cong L^{G_g/T}_{w(\lambda - \rho_n)} := G_g \times_T \C_{w(\lambda - \rho_n)}   \to G_g/T
\]
as $G_g$-equivariant, complex line bundles.
%
%
We find that
\beq{eq ind Gg}
\ind_{G_g}(\bar \partial_{L_{\lambda - \rho_n}|_{G_g w/T}} + \bar \partial_{L_{\lambda - \rho_n}|_{G_g w/T}}^*) =
\ind_{G_g}(\bar \partial_{L^{G_g/T}_{w(\lambda - \rho_n)}} + \bar \partial_{L^{G_g/T}_{w(\lambda - \rho_n)}}^*),
\eeq
where the Dolbeault operator on the right is defined with respect to $J^{G_g}_w$. 

In the proof of Proposition 5.2 in \cite{HW2}, it is shown that the right hand side of \eqref{eq ind Gg} equals $(-1)^{\dim(G_g/K_g)/2}\DInd_{K_g}^{G_g}[W^{K_g}_{w(\lambda - \rho_n + \rho_{n, g})}]$, where $\DInd_{K_g}^{G_g}$ denotes Dirac induction, $\rho_{n, g}$ is half the sum of the noncompact roots in $R^+(G_g, T)$, and $W^{K_g}_{w(\lambda - \rho_n + \rho_{n, g})}$ is the irreducible representation of $K_g$ with highest weight $w(\lambda - \rho_n + \rho_{n, g})$ with respect to the positive root system $w R^+(K, T) \cap R(K_g, T)$.
%
If $\lambda - \rho_n + \rho_{g}$ 
is regular for $G_g$, then Lemma 5.4 in \cite{HW2} implies that 
\[
\DInd_{K_g}^{G_g}[W^{K_g}_{w(\lambda - \rho_n + \rho_{n, g})}]  = [\pi^{G_g}_{w(\lambda + \rho_g-\rho_n)}],
\]
the $K$-theory class of the 
discrete series representation  of $G_g$ with Harish-Chandra parameter $w(\lambda + \rho_g-\rho_n)$. Hence \eqref{eq tau e Gg} follows, because the value of $\tau_e$ on a $K$-theory class of a discrete series representation is the formal degree of that representation \cite{Atiyah77, Connes82, Lafforgue02}.

If $\lambda - \rho_n + \rho_{g}$ is singular for $G_g$, then $d^{G_g}_{w(\lambda + \rho_g - \rho_n)} = 0$. 
Then the Dirac operator whose $G_g$-equivariant index is $\DInd_{K_g}^{G_g}[W^{K_g}_{w(\lambda - \rho_n + \rho_{n, g})}] $ has trivial kernel by Theorem 9.3 in \cite{Atiyah77}. And $\tau_e(\DInd_{K_g}^{G_g}[W^{K_g}_{w(\lambda - \rho_n + \rho_{n, g})}])$ equals the $L^2$-index of this Dirac operator (see Proposition 4.4 in \cite{Wang14}), which is defined in terms of the projection onto its kernel, and hence zero. So now the left hand side of \eqref{eq tau e Gg} is zero as well, and \eqref{eq tau e Gg} holds.
\end{proof}

\begin{theorem}\label{thm ss contr}
Let $g \in G$ be semisimple. Then $\tau_g(\ind_G(D_W)) = 0$ if $g$ is hyperbolic. If $\rank(G) = \rank(K)$ and $g$ is elliptic, then 
\begin{multline} \label{eq ss contr}
\tau_g(\ind_G(D_W)) = 
(-1)^{(\dim(G/K) + \dim(G_g/K_g))/2} \cdot \\
\frac{\sum_{w \in W_{K_g} \backslash W_K}  \det(w) e^{w(\lambda + \rho_c - \rho_g)}(g) d^{G_g}_{w(\lambda+ \rho_g-\rho_n)}}{e^{\rho- \rho_g}(g)\prod_{\alpha \in R^+(G, T) \setminus R^+(G_g, T)} (1-e^{\alpha}(g)) }.
\end{multline}
\end{theorem}
\begin{proof}
The case for hyperbolic $g$ follows from Theorem \ref{thm fixed Dolb}(a).

Suppose that $\rank(G) = \rank(K)$ and  that $g$ is elliptic. 
The element $\lambda-\rho_n \in i\kt^*$ is analytically integral as noted in Lemma \ref{lem ds Gg}. 
 In the proof of Proposition 5.2 in \cite{HW2}, it is shown that
\[
\ind_G(D_W) = (-1)^{\dim(G/K)/2} \ind_G(\bar \partial_{L_{\lambda - \rho_n}} +\bar \partial_{L_{\lambda - \rho_n}} ^* ).
\]
(This is a place where we use the assumption that $\rank(G) = \rank(K)$.)
By this equality and Lemmas \ref{lem fixed GT}, \ref{lem norm triv} and \ref{lem det N}, Theorem \ref{thm fixed Dolb}(c) implies that 
\begin{multline*}
\tau_g(\ind_G(D_W)) = 
(-1)^{\dim(G/K)/2} \cdot \\
\frac{\sum_{w \in W_{K_g} \backslash W_K}  \det(w) e^{w(\rho-\rho_g)}(g) (\tr(g|_{L_{\lambda - \rho_n}})|_{G_g w/T})  \tau_e(\ind_{G_g}(\bar \partial_{L_{\lambda - \rho_n}|_{G_g w/T}} + \bar \partial_{L_{\lambda - \rho_n}|_{G_g w/T}}^*))}{e^{ \rho- \rho_g}(g)\prod_{\alpha \in R^+(G, T) \setminus R^+(G_g, T)} (1-e^{\alpha}(g)) }.
\end{multline*}

For all $w \in N_K(T)$,
Lemma \ref{lem Rw Lmu} implies that 
\[
\tr(g|_{L_{\lambda - \rho_n}})|_{G_g w/T}  = e^{w(\lambda - \rho_n)}(g).
\]
So \eqref{eq ss contr} follows by Lemma \ref{lem ds Gg}.
\end{proof}

\begin{example}\label{ex g central}
If $g$ is central, then Theorem \ref{thm ss contr} reduces to
\[
\tau_g(\ind_G(D_W)) = e^{\lambda - \rho_n}(g)
 d^{G}_{\lambda+ \rho_c}.
 \]
 This is Proposition 4.2 in \cite{BM83}. If $g=e$, then versions of  this equality were obtained in \cite{Connes82, Lafforgue02}. This case
 was used in the proof of Theorem \ref{thm ss contr}.
\end{example}

\begin{example}\label{ex dense powers}
If the powers of $g$ are dense in $T$, then $G_g = K_g = T$, and for all $w \in W_K$,
\[
\begin{split}
\tr(g|_{L_{\lambda - \rho_n}})|_{G_g w/T}&= e^{w(\lambda - \rho_n)}(g);\\
\tau_e(\ind_{G_g}(\bar \partial_{L_{\lambda - \rho_n}|_{G_g w/T}} + \bar \partial_{L_{\lambda - \rho_n}|_{G_g w/T}}^*))&= 1.
\end{split}
\]
So Theorem \ref{thm ss contr} reduces to
\[
\tau_g(\ind_G(D_W)) = (-1)^{\dim(G/K)/2}
\frac{\sum_{w \in  W_K}  \det(w) e^{w(\lambda + \rho_c)}(g)  }{e^{  \rho}(g)\prod_{\alpha \in R^+(G, T)} (1-e^{-\alpha}(g)) }.
\]
This is Theorem 3.1(a) in \cite{HW4}.
\end{example}

\begin{remark}
If $\rank(G)>\rank(K)$, then
Theorem 3.2(a) in \cite{HW4} implies that the number  $\tau_g(\ind_G(D_W))$ equals zero for all hyperbolic $g$ and all elliptic $g$ with powers dense in a maximal torus. Then also $\tau_e(\ind_G(D_W))=0$, as in the proof of Lemma \ref{lem ds Gg}. We expect that  $\tau_g(\ind_G(D_W))=0$ for all semisimple $g$ in this case, but we have not worked out the details. That would imply that all numbers in Theorem \ref{thm index ss} equal zero if $\rank(G)>\rank(K)$.
\end{remark}

\subsection{Higher orbital integrals}


%

We no longer assume that $G$ has a compact Cartan subgroup. 
We saw in Theorem \ref{thm ss contr} that $\tau_g(\ind_G(D_W)) = 0$ if $\rank(G) > \rank(K)$. However,  $\ind_G(D_W) \in K_0(C^*_r(G))$ is nonzero, and one way to detect its nonvanishing is to pair it with the {higher orbital integrals} from \cite{ST19}. We can compute the resulting pairings  using the results in \cite{HST20} to generalise Theorem \ref{thm ss contr}. We use notation and assumptions as in Theorem \ref{thm index ss higher}.
%
%
%
As in \cite{HST20} and in Theorem \ref{thm index ss higher}, we assume that the map
\eqref{eq tilde Ad} descends to $K$.

%
\begin{corollary}\label{cor ind GT higher}
If $P$ is not a maximal cuspidal parabolic subgroup, or $g \in M$ is hyperbolic, then
\[
\langle \Phi_{P, g}, \ind_G(D_W)\rangle = 0.
\]
If $P$ is a maximal cuspidal parabolic subgroup and $g \in T \subset M$, then
\begin{multline} \label{eq ss contr higher}
\langle \Phi_{P, g}, \ind_G(D_W)\rangle  = 
(-1)^{(\dim(M/(K \cap M)) + \dim(M_g/(K \cap M_g)))/2} \cdot \\
\sum_{U \in \widehat{K \cap M}} m_U
\frac{\sum_{w \in W_{K \cap M_g} \backslash W_{K \cap M}}  \det(w) e^{w(\lambda_U + \rho^M_c - \rho^M_g)}(g) d^{M_g}_{w(\lambda_U+ \rho^M_g - \rho^M_n)}}{e^{\rho^M- \rho^M_g}(g)\prod_{\alpha \in R^+(M, T) \setminus R^+(M_g, T)} (1-e^{\alpha}(g)) }.
\end{multline}
\end{corollary}
\begin{proof}
The first statement is one conclusion of Theorem 2.1 in \cite{HST20}.

Now suppose that  $P$ is a maximal cuspidal parabolic subgroup and $g \in T$. Let $\{X_1, \ldots, X_s\}$ be an orthonormal basis of $\kp \cap \km$. For every finite-dimensional  representation $Z$ of $K \cap M$,  consider the Dirac operator
\[
D^M_{Z} = \sum_{j=1}^s R(X_j) \otimes c(X_j) \otimes 1_{Z}
\]
on $(C^{\infty}(M) \otimes \Delta_{\kp \cap\km} \otimes Z)^{K \cap M}$, analogous to \eqref{eq DW}. 
Let $\tau_g^M$ be the orbital integral with respect to $g$ as an element of $M$. Then
(5.4) and (5.5) in \cite{HST20} imply that
\begin{multline*}
\langle \Phi_{P, g}, \ind_G(D_W)\rangle  = \\
\tau^M_g\bigl(\ind_M(D^M_{\Delta^+_{\kk/(\kk \cap \km)}  \otimes W|_{K \cap M}}   ) \bigr)
-
 \tau^M_g\bigl(\ind_M(D^M_{\Delta^-_{\kk/(\kk \cap \km)}  \otimes W|_{K \cap M}}   ) \bigr).
\end{multline*}
Hence \eqref{eq ss contr higher} follows by \eqref{eq decomp Delta} and Theorem \ref{thm ss contr}.
\end{proof}

\begin{example}
If $\rank(G) = \rank(K)$, then $G$ is a (maximal) cuspidal parabolic subgroup of itself, and $\Phi_{G,g} = \tau_g$. So \eqref{eq ss contr higher} reduces to \eqref{eq ss contr}.
\end{example}

\begin{example}\label{ex central higher}
If $P$ is a maximal cuspidal parabolic subgroup and
$g$ lies in the centre of $M$, then analogously to Example \ref{ex g central}, the equality \eqref{eq ss contr higher} reduces to
\[
\langle \Phi_{P, g}, \ind_G(D_W)\rangle  = \sum_{U \in \widehat{K \cap M}} m_U e^{\lambda_U - \rho^M_n}(g)
d^{M}_{\lambda_U+ \rho^M_c}.
\]
\end{example}

\begin{example} 
If $P$ is a maximal cuspidal parabolic subgroup,  $g \in T$ and the powers of $g$ are dense in $T$, then analogously to Example \ref{ex dense powers}, the equality \eqref{eq ss contr higher} reduces to
\begin{multline*}
\langle \Phi_{P, g}, \ind_G(D_W)\rangle  = \\
(-1)^{\dim(M/(K\cap M))/2} \sum_{U \in \widehat{K \cap M}} m_U 
\frac{\sum_{w \in  W_{K \cap M}}  \det(w) e^{w(\lambda_U + \rho^M_c)}(g)  }{e^{  \rho^M}(g)\prod_{\alpha \in R^+(M, T)} (1-e^{-\alpha}(g)) }
= \\ (-1)^{\dim(M/(K\cap M))/2} \frac{\chi_{\Delta_{\kk/(\kk \cap \km)}  } (g)  \chi_W(g) }{\chi_{\Delta_{\kp \cap \km}  } (g) },
\end{multline*}
where $\chi$ denotes the character of a finite-dimensional representation. This is (6.3) in \cite{HST20}.
\end{example}

\subsection{Proofs of the index theorems}\label{sec pf ind thm}

Let $F \subset \widehat{K \times K}$ be the set of  irreducible subrepresentations of $\End(\Delta_{\kp} \otimes W)$. 
For any continuous  trace $\tau$ on $\calS_F(G)$ (with notation as above Corollary \ref{cor Rd trace F}), Theorem \ref{thm def index} implies that
\beq{eq trace index tau}
\tau^{\Delta_{\kp} \otimes W} \circ \Psi_*(\ind_{S, G}(D_W)) = 
\tau^{\Delta_{\kp} \otimes W}  (\Psi(e^{-tD_W^-D_W^+})) -
\tau^{\Delta_{\kp} \otimes W}  (\Psi(e^{-tD_W^+ D_W^-})). 
\eeq
If $g \in G$ is semisimple and $\tau= \tau_g$, then 
Proposition 3.6 in \cite{HW2} implies that the right hand side equals $\tau_g(\ind_G(D_W))$. 
So parts (b) and (c) of Theorem \ref{thm index ss} follow from Theorem \ref{thm ss contr}. Part (a) follows by Example \ref{ex g central}.

Similarly, in the setting of Theorem \ref{thm index ss higher}, we have
\[
\langle \Phi^{\Delta_{\kp} \otimes W}_{P, g}, \ind_{S, G}(D_W) \rangle = 
\langle \Phi_{P, g}, \ind_{G}(D_W) \rangle. 
\]
So Theorem \ref{thm index ss higher} follows from Corollary \ref{cor ind GT higher} and Example \ref{ex central higher}.

To prove Theorem \ref{thm index non-ss}, we just apply the computations in \cite{BM83} to the right hand side of \eqref{eq trace index tau}.
The case of part (a) of Theorem \ref{thm index non-ss} where $G \not= \SU(2n,1)$, and case (b) of this theorem, are (4.5.3) in \cite{BM83}, which comes from \cite{Barbasch79}. 
In the setting of 
part (a) of Theorem \ref{thm index non-ss}  for $G = \SU(2n,1)$, the equality  (4.5.6) in \cite{BM83} states that  the right hand side of \eqref{eq trace index tau} equals the right hand side of \eqref{eq ind thm taun}. Part (c) of Theorem \ref{thm index non-ss} is Proposition 5.6 in \cite{BM83}. (Note that there is a typo in  (7.1.5) in \cite{BM83}: $c_{\mu}(\sigma)>0$ should be $c_{\mu}(\sigma)=0$.)

In the setting of part (d) of Theorem  \ref{thm index non-ss} with $\lambda + \rho_c$ regular,  the right hand side of \eqref{eq trace index tau}, with $\tau = \tau_{\rem}$, is computed on page 194 of \cite{BM83}.
There is a slight error here, however, as pointed out in \cite{Stern90}: the computation of the term involving $W_{\Gamma}$ in \cite{BM83} should be modified as in (6.5) in \cite{Stern90}. This leads to \eqref{eq ind thm T}.
The case of part (d) of Theorem  \ref{thm index non-ss} for real hyperbolic space of dimension at least $4$ is proved on the same page of \cite{BM83}.

 \bibliographystyle{plain}

\bibliography{mybib}

\end{document}